 \documentclass[11pt]{article}
\usepackage{amsmath, amsthm, amsfonts, amssymb}
\usepackage{color}

\hoffset= - 0.5 cm
 \voffset= - 0.1 cm

\numberwithin{equation}{section}

\newtheorem{theorem}{Theorem}[section]
\newtheorem{lemma}[theorem]{Lemma}
\newtheorem{proposition}[theorem]{Proposition}
\newtheorem{corollary}[theorem]{Corollary}
\newtheorem{remark}[theorem]{Remark}
\newtheorem{definition}[theorem]{Definition}
\newtheorem{hypothesis}[theorem]{Hypothesis}

\newcommand{\rr}{{\mathbb R}}
\newcommand{\nn}{{\mathbb N}}
\newcommand{\zz}{{\mathbb Z}}
\def\P{{\mathbb P}}
\def\Q{{\mathbb Q}}
\def\E{{\mathbb E}}

\def\hh{{\vskip 1mm \noindent}}

\begin{document}

\title {Uniqueness in Law for Stochastic Boundary Value Problems
  }

\author{Anna Capietto
 \footnote
 {e-mail: \ \ anna.capietto@unito.it.
 Supported by
 the M.I.U.R. research
  project  2007 ``Topological
 and Variational Methods in the Study of Nonlinear Phenomena''.}\
 \ \ \ \& \ \  \
 Enrico Priola \footnote
 {e-mail: \ \ enrico.priola@unito.it. Supported by
  the M.I.U.R. research project 2006
   ``Kolmogorov Equations''.}
\\
\\
{\small  \it Dipartimento di Matematica, Universit\`a di Torino}
\\
{\small  \it   Via Carlo Alberto 10,   10123 Torino, Italy. }}


\maketitle


{\vskip 7 mm }


{\vskip 2 mm }\noindent {\bf Key words:}  stochastic boundary value
problems, anticipative Girsanov theorem, uniqueness in law.

\vspace{2 mm}

\noindent {\bf Abstract:}  We study existence and uniqueness of
 solutions for second order ordinary  stochastic differential
equations
 with Dirichlet boundary conditions on a given interval.
 In the first part of the paper we provide sufficient conditions
   to ensure pathwise
 uniqueness, extending some known results.
 In the second part we show sufficient conditions
  to have the weaker concept of uniqueness in law and provide
   a significant  example. Such conditions
   involve   a linearized equation
   and  are of different type with respect
 to the ones which are usually imposed   to
  study   pathwise uniqueness.
 This seems
 to be the first paper which deals with uniqueness in law for
 (anticipating) stochastic boundary value problems.
    We mainly
  use functional  analytic
  tools  and some concepts of Malliavin Calculus.

\section {Introduction }

 \noindent The object of
this paper is the  stochastic ordinary differential equation
\begin{equation}\label{np}
 \frac{d^2 X_t} {dt^2}  + f\Big (t,X_t, \frac{d X_t} {dt}
 \Big) = \frac{d W_t}{dt}, \qquad t
\in [0,1],
\end{equation}
subject to the boundary condition $X_0=0=X_1$, where $f:[0,1]\times
\rr^2\to\rr$ is a given continuous function and $(W_t)$ is a
 one-dimensional Wiener process starting from 0 (note that
 $X_t =X(t) , $ $t \in [0,1])$.

\noindent There is a wide literature on (anticipating) stochastic
boundary value problems (see, for instance, \cite{AN}, \cite{D},
\cite{DM}, \cite{NP}, \cite{NP1},
 \cite{OP}). Methods for numerically
solving stochastic boundary value problems are investigated as well
(see \cite{A} and the references therein). Usually, once the
existence of a solution $X= (X_t)$ is guaranteed, the question of
uniqueness is tackled in the {\it pathwise} sense
 (i.e., if $Z $ is another solution to \eqref{np}, then
  $X = Z$, $\P$-a.s., where $\P $ denotes the Wiener measure on
   $C_0([0,1])$, see Section 2).
    Having in mind an application of the contraction principle, it is
 usually required, roughly speaking when $f(t,x,y)= f(x)$,
  that $f$ is globally  Lipschitz with a Lipschitz
  constant
    small enough or that $f$ satisfies a
    kind of monotonicity condition.

\noindent Our contribution is two-fold. On one hand, concerning
pathwise uniqueness, we show in Section 3 that some of the methods
of nonlinear analysis (see  the seminal work
  \cite{MW} and the book \cite{CH})
   for deterministic ordinary differential equations are
suitable for improving some of the results already available
 in the
 literature.
  On the other hand, we propose a new step in the study of
stochastic BVPs, i.e. we provide sufficient conditions for
 the weaker  concept of
   uniqueness {\it in law} of  solutions (i.e.,
    if $Z $ is another solution to \eqref{np}, then
 $\P (X^{-1}(A)) = \P (Z^{-1}(A))$, for any Borel set $A
  \subset C_0([0,1])$).
  Such conditions are of different type w.r.t. the available results on
  pathwise uniqueness (see, in particular Section 4.5 and
   also  Section 4.6, which contains a significant example).
    Roughly speaking, our Theorem
  \ref{Nonres1} in Section 4.6   shows that uniqueness in law holds
   even if a ``typical'' non-resonance condition is violated on a discrete
  set of points.  On the other hand,  we do not know
   if pathwise uniqueness
 holds in such a case, since the usual methods of nonlinear
  analysis fail.

\noindent Note that, to the authors'
  knowledge, up to now uniqueness in law has been treated only for
 the well-studied
 (non-anticipating) Cauchy problem for stochastic differential equations
 (cf., for instance, \cite{IW}).

\noindent We first concentrate on a precise definition of the notion
 of solution. Indeed, according to
  the paper \cite{NP} by Nualart-Pardoux (which was the
starting point of our research), we understand \eqref{np} in the
integral sense, i.e., we require that $X: C_0([0,1]) \to C_0([0,1])$
is Borel
 measurable and
  that
 (setting $\frac{dX_t}{dt}= X'_t$)
\begin{eqnarray}\label{int}
X_t^\prime (\omega) + \int_0^t f(s,X_s(\omega),X_s^\prime(\omega))
ds  = X_0^\prime(\omega) + \omega_t, \qquad t \in [0,1],
\end{eqnarray}
for any $\omega \in C_0([0,1])$, $\P$-a.s. (see Section 2
 for the precise definition).
 Then
 existence and pathwise uniqueness
 of solutions to \eqref{int}
  are investigated, arguing for
   a fixed   $\omega
 \in C_0([0,1])$.
  In Section 3.1 we use the global implicit function theorem and
provide an existence and uniqueness result (Theorem \ref{Nonres})
under a non-resonance type condition; this goal is reached after
writing \eqref{int} as an abstract equation involving the Green's
function of $-d^2/dt^2$ (with Dirichlet boundary condition). In
Section 3.2 we give sufficient conditions (of Lipschitz type) on
$f(t,x,y)$ which enable us to study the BVP  (\ref{np}) as a fixed
point problem and to apply the contraction mapping principle. In
particular, Corollaries \ref{PrimoCor} and  \ref{SecondoCor}
 improve  related
 results in  \cite[Section 1]{NP}. Section 3 ends with a
discussion on the Fredholm alternative for \eqref{np}.

\noindent Once this first aspect has been developed, it is quite
natural to consider the case in which pathwise uniqueness is not
guaranteed (see Section 4). To this purpose,
 we deal with   the mapping
 $T: C_0([0,1]) \to C_0([0,1])$ introduced
 in \cite{NP}:
$$
T_t(\omega) = \omega_t + \int_0^t f (s, Y_s (\omega),  Y_s'
(\omega)) ds,\;\; \omega \in C_0([0,1]), \; t \in [0,1],
$$
 where  $Y = (Y_t)$ is the solution  to
\eqref{np} corresponding to $f=0$. In \cite{NP}
 it is  shown that if
 $T$ is {\it bijective} then existence and pathwise uniqueness hold
for \eqref{np} (see also Proposition \ref{serve}).
  We first show  that even if
 $T$ is not bijective, there always exists a measurable {
  \it left inverse
$S$ } of
  $T$ provided that  a solution $X$ exists (see Lemma
\ref{1}).
   This was our starting point to study uniqueness in law.
   Indeed, once the existence
of a left inverse is proved
  the aim is to use a non-adapted version
 of the Girsanov theorem recently proved
  by \"Ust\"unel-Zakai in \cite{UZ} (see Section 4.2).

\noindent Remark that to study
 uniqueness in law we can not use the well known
 non-adapted version of the Girsanov theorem due to Ramer
 and Kusuoka (see \cite{K},  \cite{R},
  and also \cite[Section 4.1]{N}).   This result has been
  already applied  to stochastic BVPs  in \cite{D}, \cite{DM} and
 \cite{NP},  in order to investigate the
  Markov
  property when a unique solution exists.   The Ramer-Kusuoka
  theorem    would require that {\it $T$ is bijective}
  (i.e., pathwise uniqueness holds for \eqref{np}).
   This is not the case for
   the Girsanov theorem in \cite{UZ} which,  however,
   requires some  additional
    hypotheses (involving Malliavin
   Calculus)  which are not present in \cite{K} and \cite{R}.

 \noindent   Although the  formulation of \cite[Theorem 3.3]{UZ}
   involves  Sobolev spaces of Malliavin Calculus,
   we find more useful to  deal with
 the strictly
  related notion of  $H$-differentiability (cf. Section 4.1 and see \cite[Section 4.1.3]{N} and \cite{S}).
 By using the inverse function
 theorem and some functional analytic tools, we first show the
 $H$-differentiability of the transformation
  $F : \Omega \to \Omega$,
 $$
F_t(\omega)  = - \int_0^t f(s, X_s (\omega), X_s'(\omega))ds,
 \;\; \omega \in \Omega, \; t \in [0,1],
$$
 where $X$ is a given solution (see Theorem \ref{2});
  it turns out that
  $S = I + F$ is the above mentioned  left inverse of
  $T$.  Then we prove
  an
exponential
 estimate for the   Skorohod integral of $F$ (see Section
4.4) which is required in the Girsanov theorem of \cite{UZ}.
 Remark that  the known exponential
  estimates
 (cf.  \cite{U} and \cite[Appendix B.8]{UZ}) are not
  applicable to get  our bound.
    We obtain
  the
  required exponential integrability
 assuming that $f$ is bounded.

\noindent In Section 4.5  we prove
 a uniqueness in law result in    the following form
  (assume  for simplicity that
 $f (t,x,y)$ $= f(x)$).
  If
 $f \in C^{2}_b (\rr)$, then
   uniqueness in law for \eqref{np} holds
    among all the solutions $X$ such that  the
 corresponding linearized equations
$$
u_t''  + a_t(\omega) u_t =0,\;\; u_0= u_1=0,
$$
where   $a_t(\omega) = f' (X^{}_t(\omega) )$, $t \in [0,1]$,
 have the only solution $u=0$, for
any $\omega \in \Omega$, $\P$-a.s.. This means that
  uniqueness
  in law holds for \eqref{np} whenever one is able to
   prove that all solutions
 $X$ to \eqref{np} verify our   assumption
  on the linearized equation.
 In Section  4.6 we show a concrete class of BVPs
 for which this is possible. Note that
  in Theorem \ref{Nonres1} of Section 4.6
    we  also establish  existence of solutions; this  is
     quite involved (see also Remark \ref{exi} where a more general
 existence result is formulated).

\noindent The previous condition on the linearized
 equation can be, roughly
speaking, interpreted (from the nonlinear analysis point of view) as
a requirement on the invertibility of the differential of the map
$S$; indeed,
  as it is explained in the proof of  Theorem \ref{2},
it ensures that $S$ is a local homeomorphism. In order to obtain a
global homeomorphism, and thus {\it pathwise} uniqueness, Section 3
shows that some additional assumptions (such as the non-resonance
condition \eqref{ipononres}) have to be added. Thus, a rough
comparison between our pathwise and ``in law''
 uniqueness results
may be proposed in the sense that the fact that $S$ is a local
diffeomorphism is sufficient to guarantee  uniqueness in law.

\noindent Finally, in Section 5 we tackle a  problem which arises
when dealing with  non-adapted versions of the Girsanov theorem.
 It consists of the determination of an explicit expression for a
 Carleman-Fredholm determinant
 related to the mapping $T$ (see  \eqref{fr})
 This expression
  is reached in \cite{D} and \cite{NP}
   with an  involved  proof
 based on Malliavin calculus. We propose
  an alternative
 shorter proof based on a functional-analytic
 approach taken from the book  \cite{GGK}. We believe that this
 method can be extended to other
     situations in which the computation
    of  Carleman-Fredholm determinants is of interest.
     We
  also use the methods of \cite[Chapter XIII]{GGK}
  to find the expression of the  Malliavin derivative of $F$ (see
   Proposition \ref{esplic}).
  An
 account of the ideas from \cite{GGK} can be found in Appendix B.

{\vskip 2mm}  \noindent  {\bf Acknowledgments}  \ The authors are
 grateful to Paolo Cermelli for many helpful and fruitful
 discussions. They also thank the collegues of the
  I.N.D.A.M.  project
 ``Does noise simplify or complicate the dynamics of nonlinear
 systems?'' for useful conversations.

\paragraph {Basic Notations} $\Omega = C_0([0,1])$ denotes the Banach space of all
real continuous functions on $[0,1]$ which vanish in $t=0$, endowed
with the supremum norm
 $\| \cdot\|_0$. Moreover,
 $\cal F$ is the Borel $\sigma$-algebra on $\Omega$ and $\P$ the
   Wiener measure
  on $\Omega$; $\P$ can be uniquely characterized by saying that
   on the  probability
  space $(\Omega, {\cal F}, \P )$, the stochastic (coordinate) process $W =
   (W_t)$,
\begin{align} \label{wie}
W_t (\omega) = \omega (t),\;\;\; \omega \in \Omega, \; t \in [0,1],
\end{align}
   is a real Wiener process (up to time $t=1$).  As usual,
     when a property
 concerning $\Omega$ holds for any $ \omega \in
 \Omega_0 $, with $\Omega_0 \in {\cal F}$ and $\P (\Omega_0)=1$,
 we
 say that this property  holds $\P$-a.s..
 \noindent The subspace  $C^1_0$  of $\Omega$ consists of all  $C^1$-functions
 vanishing at $t=0$ and $t=1$.

 \noindent Let $H_1$ and $H_2$   be real separable Hilbert spaces (with inner
product $\langle \cdot,
 \cdot \rangle_{H_k} $ and norm $|\cdot|_{H_k}$, $k=1,2$).
  A linear and bounded operator
 $L : H_1 \to H_2$ is said to be a Hilbert-Schmidt operator if
  for some orthonormal basis $(e_n)$ in $H_1$ we have
  $
 \sum_{n \ge 1} |L e_n|^2_{H_2} < \infty.
$

 \noindent The space of all Hilbert-Schmidt operators will be
 indicated with $H_1 \otimes H_2$ or ${\mathcal HS}(H_1, H_2)$;
 it is   a Hilbert space with
 the inner product   $\langle \cdot, \cdot \rangle_{H_1 \otimes
 H_2} $,
$
  \langle R, S \rangle_{H_1 \otimes
 H_2}  =  \sum_{n \ge 1} \langle R e_n, S
 e_n\rangle_{H_2}
 $
(see, for instance,
 \cite[Chapter IV]{GGK} or \cite[Chapter VI]{RS}).

  \noindent The corresponding Hilbert-Schmidt norm is indicated by $\|
\cdot\|_{H_1 \otimes
 H_2}$;  $\| \cdot \|_{\mathcal L(H_1, H_2)}$
   denotes the operator
norm in the Banach space ${\mathcal L(H_1, H_2)}$ of all bounded and
linear operators from $H_1$ into $H_2$.

 \noindent Let $K$ be a real separable Hilbert space.
   We recall that if $A$ and $B$ are linear bounded
   operators from
   $K$ into $K$ and $B$ is Hilbert-Schmidt, then
    $AB$ is also  Hilbert-Schmidt and
\begin{align} \label{hil}
\| AB\|_{K \otimes K}
     \le \| A\|_{{\mathcal L}(K, K)} \| B\|_{K \otimes K}.
\end{align}
 If $L$ is a Hilbert-Schmidt operator from $K$ into $K$, the
  {\it Carleman-Fredholm   determinant} of $I +L$  is
$$
 {\det}_2 (I + L) = \prod_{k \ge 1} (1+ \lambda_k)
 e^{-\lambda_k},
$$
where $\lambda_k$ are the eigenvalues of $L$, counted with respect
to  their multiplicity (see \cite[Appendix A.2]{UZ1} and
\cite{GGK}).

 We set  $H = L^2(0,1)$ and consider
 also   $H_0 =   \{ f \in
\Omega\, :\, $ there
 exists the distributional derivative $f' \in H \}.$
  It is well
  known that  any  $f \in H_0$   is
  absolutely continuous and so differentiable  a.e., with
   the  derivative   defined
   a.e. which  coincides with the distributional derivative.

 The  space  $H_0 $
   will be  {\it  considered isomorphic} to $H$
  and so identified (when no confusion may arise)
   with $H$ through the isomorphism
   $f \mapsto f'$ from $H_0$ onto $H$; its   inverse mapping will
 be simply denoted by $  \sim  $, i.e.,
 $
 \tilde f_t = (\tilde f)_t
 = \int_0^{t} f_s ds,  \;\; $ $ f \in H,$ $ \;\; t \in [0,1].
$  By  defining the inner product
$$ \langle h,g \rangle_{H_0} :=
\langle h',g' \rangle_H,
  \;\;\; f, g
\in H_0,
$$
 $H_0$ becomes a real separable Hilbert space.

\section {Preliminary results }

In this section we introduce the basic boundary value problem
studied in later sections, and give two equivalent integral
formulations of it.

\medskip
  \noindent Let
$
 f:[0,1]\times \rr^2\to\rr  \;\; \text{be a given  continuous
 function}.
$

\noindent  An  Borel  set $\Omega_0 \subset \Omega$
 is called {\it admissible}
 if $\P (\Omega_0) =1$ and, moreover, for any
 $\omega \in \Omega_0 $, $\P$-a.s.,
 for any
  $h \in H_0$, we have that $\omega+ h \in \Omega_0 $ (i.e.,
   $\omega  + H_0 \subset \Omega_0 $, for any $\omega \in \Omega_0
    $, $\P$-a.s.).

 \noindent A {\it Borel measurable}  mapping
  $X : \Omega \to
\Omega$, $X = (X_t)$, $t \in [0,1]$,  is said to be a
 {\it   solution}
   of \eqref{np} if  there exists an {\it admissible open set}
    $ \Gamma \subset \Omega$, such that
   $X (\omega) \in C^1_0 $,
 for any  $\omega \in \Gamma$,
 and, for any $t \in [0,1]$, we have
$$
X_t^\prime (\omega) + \int_0^t f( s,X_s(\omega) , X_s^\prime(\omega)
) ds = X_0^\prime(\omega) + \omega_t; \;\;
X_0(\omega)=X_1(\omega)=0, \;\; \omega \in \Gamma.
$$
  We say that {\it pathwise uniqueness} holds for \eqref{np}
 if given two solutions $X$ and $Z$, we have $X = Z$, $\P$-a.s.;
  we say that {\it uniqueness in law } holds for \eqref{np}
 if given two solutions $X$ and $Z$, they have the same law,
 i.e., for any  $A \in {\cal F} $,
  we have $ \P (X \in A) := \P (X^{-1}(A) )= $ $\P (Z \in A).$

  \noindent In the sequel we will often omit  dependence on $\omega$
 of $X$ and write,
more shortly,
\begin{eqnarray}\label{integro}
X_t^\prime+ \int_0^t f( s,X_s ,  X_s^\prime ) ds = X_0^\prime +
\omega_t, \qquad X_0=X_1=0,\;\; \omega \in \Gamma,
\; t \in [0,1].
 \end{eqnarray}

 \begin{remark}
{\rm Pathwise uniqueness is investigated by \cite{NP} always
assuming $\Gamma=\Omega$; our generality
 is also motivated by the existence and uniqueness
  result in Section 4.6.
}
\end{remark}

 \noindent An easy equivalence between the classical and weak formulation
 of solutions is proved in the next result.

\begin{proposition} \label{equivalenti}
 A  Borel measurable
mapping
  $X : \Omega \to
\Omega$ such that $X (\omega) \in C^1_0,$
 for any  $\omega \in \Gamma $ ($\Gamma$
    is an admissible  open set in $\Omega$),
 is  a
    solution
 if and only if it satisfies,
 for every $\varphi \in C^1_0$,
\begin{eqnarray}\label{weak}
- \int_0^1 \varphi_t^\prime X_t^\prime  dt +  \int_0^1 \varphi_t
f(t,X_t,X_t^\prime) dt
 + \int_0^1 \varphi_t^\prime \omega_t  dt = 0, \;\; \omega \in
 \Gamma.
\end{eqnarray}
\end{proposition}

\begin{proof} We have to show equivalence between \eqref{integro}
and \eqref{weak}.  It is clear that
 if $X $ is  a solution according to  (\ref{integro})
  then  (multiplying by $\varphi \in C^1_0$ and
    integrating by parts)
   $X$ is also a solution to
   (\ref{weak}).

 \noindent Let now  $X$  be a solution according to (\ref{weak}).
Letting
$$
 u_t = X_t^\prime + \int_0^t f(s,X_s,X_s^\prime) ds  -
X_0^\prime - \omega_t,\qquad t \in [0,1], $$
we obtain $ \int_0^1 u_t \psi_tdt =0, $   for every $\psi \in
C([0,1])$ with zero mean. This means that
$$ \int_0^1 u_t \sin (2\pi n t ) dt =\int_0^1 u_t \cos (2\pi n t )
dt =0, \qquad n \ge 1. $$
By the  $L^2$-theory of Fourier series,  $u $ is a.e. constant; but
since $u_0=0$ and $u$ is continuous it must be  $u_t = 0$ for every
$t \in [0,1]$;  it follows that $X_t$ is a solution of
(\ref{integro}). Alternatively, to prove that $u$
 is constant, one can use \cite[Lemma VIII.1]{B}.
\end{proof}

 \noindent Following \cite{NP}, we consider the solution  $Y $ to \eqref{np}
 corresponding to $f=0$, i.e.,
\begin{align}\label{y1}
Y_t (\omega) = - t  \int_0^1  \omega_s ds +  \int_0^t \omega _s ds ,
\qquad t \in [0,1], \; \omega \in \Omega.
\end{align}
Note that $Y : \Omega \to \Omega$ is a linear continuous and one to
one mapping; moreover $Y(\Omega) = C^1_0$. Moreover, if $Y (\omega)
= \eta$ then
 $Y_t^{-1}(\eta)=\omega_t= \eta_t' - \eta_0' $, $t \in
 [0,1]$.

 \noindent Proposition \ref{equivalenti} allows us to rewrite the
 boundary value problem (\ref{np}) as an integral equation. Consider
in fact the Green's function of $-d^2/dt^2$ (with Dirichlet boundary
condition)
\begin{equation}
K(t,s)=t\wedge s-ts.
\end{equation}
First note that
\begin{equation}\label{g} Y_t(\omega)=\int_0^1\frac{\partial
K}{\partial s}(t,s)\omega_s ds, \;\; t \in [0,1].
\end{equation}
Equivalently, using the
 stochastic It\^o integral, we have,
$\P$-a.s.,  $Y_t =-\int_0^1 K(t,s) d\omega_s$,
 $t \in [0,1]$.

 \noindent Introducing the
operator
\begin{equation} \label{kappa}
{\cal K}: \Omega \to \Omega  \qquad\quad  v \mapsto \int_0^1
K(\cdot, s) v_s ds,
\end{equation}
we have the following standard result, whose proof is omitted for
brevity (see also   \cite{D}).

\begin{lemma}\label{equiv2}
 A   measurable  mapping
  $X : \Omega \to
\Omega$, such that $X (\omega) \in C^1_0,$
 for any  $\omega \in \Gamma$
  $(\Gamma$  is an admissible  open set  in $\Omega$)
  is  a
    solution of \eqref{np}
     if and only if it solves the integral equation
\begin{equation} \label{green}
X(\omega)-{\cal K}(f(\cdot, X(\omega),X^\prime(\omega))) =
Y(\omega), \;\; \omega \in \Gamma.
\end{equation}
\end{lemma}
 \noindent Lemma $\ref{equiv2}$ shows that the existence  of solution to
$(\ref{np})$ is equivalent to the existence of a fixed point for the
operator $X \mapsto {\cal K}(f(\cdot, X,X^\prime))+ Y(\omega)$,
 for any
$\omega \in \Gamma$;
  such fixed point must also depend
measurably on $ \omega$. By the properties of the Green's function,
if $X = X(\omega)$  is a fixed point of this operator then
necessarily $X_0=0=X_1$.

\medskip  \noindent As in \cite{NP} let us introduce the  operator
 $T : \Omega \to \Omega$,
 \begin{align} \label{t1}
T_t(\omega) = \omega_t + \int_0^t f (s, Y_s (\omega),  Y_s'
(\omega)) ds,\;\; \omega \in \Omega, \; t \in [0,1].
\end{align}
Note that $T$ is continuous on $\Omega$.

 \noindent The following useful result is an
  extension of  \cite[Proposition 1.1]{NP}.
  It characterizes pathwise uniqueness for
   \eqref{np} by means of the mapping $T$.
     We provide a proof for the
 sake of completeness.


\begin{proposition} \label{serve}
 The following assertions are equivalent.

(i) There exists  an admissible open set   $\Gamma \subset
 \Omega$   such that
 the mapping $T: T^{-1}(\Gamma)
   \to \Gamma$ is bijective.

(ii) There exists  an admissible open set   $\Gamma \subset
 \Omega$, such that,
   for any $\omega \in \Gamma $,
 there exists a unique  function
 $u \in C^1_0 $ which is a solution of
\begin{equation} \label{vecc}
\left\{\begin{array}{l}
           u_t^\prime +
           \int_0^t f(s,u_s,u_s^\prime)  ds  = u_0^\prime  +
\omega_t\\
            u_0=0=u_1.
            \end{array}
\right.
\end{equation}
Moreover, if (i) (or (ii)) holds, then there exists
  a pathwise  unique
 solution
 $X $ to \eqref{np} which is given by
 $X(\omega)= Y (T^{-1}(\omega))$, $\omega \in \Gamma$
 and $X(\omega) =0$ if $\omega \in
  \Omega \setminus \Gamma$.
\end{proposition}
\begin{proof}
 $(i) \Longrightarrow (ii)$. We first show the existence of a
solution $u$ corresponding to $\omega \in \Gamma$. Let $\eta=
T^{-1}(\omega)$ and define $u:= Y (T^{-1}(\omega))$. We find, for $t
\in [0,1]$,
$$
u'_t = Y_t'(\eta) = - \int_0^1 \eta_s ds + \eta_t = Y_0'(\eta)
 + \omega_t - \int_0^t f (s, Y_s (\eta),  Y_s'
(\eta)) ds$$
$$
= u'_0 + \omega_t - \int_0^t f (s, u_s ,  u_s' ) ds.
$$
Uniqueness is obtained  from the injectivity of $T$, using the
following fact: if $u \in C^1_0$ is any solution to \eqref{vecc}
 with $\omega \in \Gamma$,
then we have $T(Y^{-1}(u) ) = \omega $ (see the comment after
\eqref{y1}).

 \noindent $(ii) \Longrightarrow (i)$. Let us check that $T$ is onto.
 For a fixed $\omega \in \Gamma$, let $u$ be the solution corresponding to
 $\omega$.
   We define $\eta_t = Y^{-1}_t (u) = u_t' - u_0'$, $t \in [0,1]$.
    We immediately find $T (\eta ) = \omega.$
Let us verify that $T$ is one to one. If $\eta = T(\omega_1)
 = T(\omega_2)$, then we have, for $k=1,2,$
$$
\eta_t = \omega_k (t) + \int_0^t f (s, Y_s (\omega_k),  Y_s'
(\omega_k)) ds,\;\; t \in [0,1].
$$
Since $\omega_k (t)=
  Y_t' (\omega_k) - Y_0' (\omega_k) $, $t \in [0,1]$,
 $k=1,2$,
   we see that $u^{(1)} = Y (\omega_1)$ and $u^{(2)} = Y (\omega_2)$
 are two solutions to \eqref{vecc}
  (when $\omega = \eta$). It follows that
$Y (\omega_1) = Y (\omega_2)$ and so $\omega_1 = \omega_2$.

 \noindent To prove the final assertion, i.e., that the
  given  $X $  is in fact
  a solution, it remains to check that $X : \Gamma \to \Omega$
   is Borel measurable. Since $Y$ is continuous,  the assertion
     holds if
    $T^{-1} : \Gamma \to T^{-1}(\Gamma)$
      is measurable. To show this fact
    it is
    enough to apply an important theorem due to
  Kuratowski (see \cite{Pa}[Section 1.3]).
 This result states
 that
  any Borel measurable  mapping $\varphi$
   from a  complete separable metric space $F_1$
    into another complete separable metric space
     $F_2$,
   which is also bijective
   from a Borel subset $E_1 \subset F_1$
  onto a Borel subset $E_2 \subset F_2$,
 has the inverse $\varphi^{-1}: E_2 \to E_1$ which is  Borel
  measurable (i.e., $\varphi$ is a measurable isomorphism).
\end{proof}

\section { Pathwise Uniqueness }

In this section we adapt techniques from the classical theory of
boundary value problems to the integro-differential equation
(\ref{integro}) and obtain sufficient
 conditions on the function $f$
 which guarantee the existence
  and  pathwise uniqueness of the solution for
 {\em any} given $\omega\in \Omega $ (i.e.,
   we can
  take, as it is  done in \cite{NP},
  $\Gamma = \Omega$ in the definition of solution to \eqref{np}).

\subsection{Existence and uniqueness under non-resonance conditions}

 Consider the boundary value problem
\begin{equation}\label{eqnonres}
X_t^\prime + \int_0^t f(s,X_s)  ds  = X_0^\prime  + \omega_t,\qquad
X_0=X_1=0, \;\; \omega \in \Omega,
\end{equation}
and assume that  $f:[0,1]\times \rr \to\rr$ is
 continuous and
 differentiable with respect to its second argument with bounded
  derivative.

 \noindent  By Lemma \ref{equiv2} and Proposition \ref{serve},
  solvability of \eqref{eqnonres} is proved if,  for any
  $\omega \in \Omega$, there exists a unique
   function $u \in C^1_0$
   which satisfies
 \begin{align} \label{fo}
 u_t - {\cal K}(f(\cdot,
 u ))(t) = \int_0^1\frac{\partial
K}{\partial s}(t,s)\omega_s ds,\;\; t \in [0,1].
\end{align}
 Write $ H =L^2(0,1)$ and introduce
\begin{equation}\label{mapPhi}
\Phi:H \longrightarrow H, \qquad   u
 \mapsto  f(\cdot,
 u (\cdot) ).
\end{equation}
 Notice that the existence and uniqueness of the solution
 of (\ref{fo}) for
 every $\omega\in  \Omega$ is guaranteed,
 {\it in particular, } if the
map
 $$
 (I - {\cal K}\Phi):H \longrightarrow H, \qquad   u
 \mapsto  u - {\cal K}(f(\cdot,
 u (\cdot)))
$$
is a global homeomorphism.
  In order to apply a variant of the abstract global implicit
function theorem (cf. \cite[Theorem 3.9, page 29]{CH})
 to (\ref{mapPhi}), we
 shall need the following

\begin{lemma}\label{operatori}
 (\cite[Lemma 3.4, page 95]{CH}) Let $M$ be a real Hilbert space and
${ K}:M \to M$ be a  compact, symmetric, positive  definite
operator. Let
 $0<\lambda_1 \leq \lambda_2 \leq \dots \leq \lambda_n \leq \dots$
 be its eigenvalues (counted according to their  multiplicity).
 Consider a family ${\cal A}$ of symmetric linear operators on $M$,
and assume that there exist $\mu_n,\mu_{n+1}$, such that
\begin{equation} \label{ipooper} \lambda_n I < \mu_n I \leq A \leq
\mu_{n+1} I < \lambda_{n+1} I, \;\;\;\;  n \ge 1,
\end{equation}
for each $A \in {\cal A}.$
 Then, the linear map $F: M \to M$, $x \mapsto x
- { K} A x$, for each $A \in {\cal A}$ has a bounded inverse and
there exists $N>0$ such that
\begin{equation} \label{stimaoper}
\|(I - { K} A)^{-1}\|_{{\cal L}(M, M)}
 \leq N,   \qquad {\rm for \, all}\,\, A \in {\cal
A}.
\end{equation}
\end{lemma}

\noindent We can now state and prove the main result of this
section.

\begin{theorem}\label{Nonres}
Assume that
\begin{equation}\label{ipononres}
\pi^2m^2< h \leq \frac{\partial f}{\partial x}(t,x) \leq k <
\pi^2(m+1)^2, \quad  t\in [0,1],\; x\in \rr,
\end{equation}
where $m\geq 0$ is an integer and $h,k$ are real constants. Then
(\ref{eqnonres}) has a unique solution.
\end{theorem}

\noindent The assumption on $ \frac{\partial f}{\partial x}$ is a
non-resonance condition in the sense that zero is the only solution
to the BVP associated to the linear problem $ v_t^{''}+
\frac{\partial f}{\partial x}(\tau,\xi) v_t =  0,$ for any fixed
$\tau,\xi \in \rr$.

\begin{proof}  We only give a sketch of the proof,
 since it  is similar to the second
proof of  \cite[Theorem 3.3, page 93]{CH}. This
  proof  consists of an
  application of  \cite[Theorem 3.9, page 29]{CH} and Lemma
  \ref{operatori}. As mentioned above, we have to show that
 $(I - {\cal K} \Phi)$
 is a global homeomorphism from $H$ onto $H$.
To this end, it is sufficient to check  that $\Phi$ in
\eqref{mapPhi} is of class $C^1$
 on $H$ and
  that  $(I - {\cal K}D\Phi(u))^{-1}$ exists, for any $u \in H$
  ($D\Phi(u)$ being the Fr\'echet derivative of
   $\Phi$ at $u \in H$)
   and satisfies,
  for some $N>0,$
the inequality
\begin{equation}\label{difi}
\|(I - {\cal K} D\Phi(u))^{-1}\|_{{\cal L}(H,H)}
 \leq N,   \qquad {\rm for \, all}
\,\, u \in H.
\end{equation}
\noindent From the assumptions on $f$,
it follows that $\Phi$ is of class $C^1$. In order to verify
\eqref{difi}, it suffices to apply Lemma \ref{operatori} with $M=H,
K={\cal K}$,
 $\lambda_n=(n
\pi)^2$, taking as $\cal A$ the family of all bounded linear
 operators  on $H$ defined by $Ay(t)=
 D\Phi(u) [y](t)$ $
 = {\frac{\partial
f }{\partial x}}(t, u(t))y(t)$,  for every $u \in H$. It is clear that the non-resonance hypothesis allows us to apply
Lemma \ref{operatori}.
 \end{proof}

 \noindent We close this section  with a short discussion of the Fredholm
alternative in our context. Consider a linear BVP for which
\begin{equation}\label{linear}
f(t,X_t,X_t^\prime)=\mu X_t,
\end{equation}
with $\mu>0$  a real positive constant. By Lemma \ref{equiv2} we
know that (\ref{np}) with (\ref{linear}) is equivalent to the
linear integral equation
\begin{equation}
(I-\mu {\cal K})X(\omega)=Y(\omega), \;\;\; \omega \in \Omega,
 \label{fredholm}
\end{equation}
with $Y$ given by (\ref{g}). The operator ${\cal K}$ is self-adjoint
 in $L^2(0,1)$ and the eigenvalues and eigenfunctions of  ${\cal K}$ are
$\frac{1}{\mu}=\frac {1}{k^{2}\pi^{2}}$ and $\sin(k\pi t)$, with $k$
integer, $k \ge 1$. The classical Fredholm alternative  states that,
if $\mu\neq k^2\pi^2$, then $\ker(I-\mu{\cal K})=\{0\}$ and
(\ref{integro}) admits  a unique solution $X(\omega)=(I-\mu{\cal
K})^{-1}Y(\omega) $,
 while for $\mu=k^2\pi^2$
there exist solutions if and only if $Y$ is orthogonal in $L^{2}$ to
the eigenfunctions of ${\cal K}$.

\noindent In our case, the requirement that $Y$ be orthogonal in $L^2(0,1)$ to
the eigenfunctions of ${\cal K}$ yields
\begin{align}\label{ortho}
\nonumber \int_{0}^{1} Y_t(\omega)\sin(k\pi t)\,dt =-
\int_{0}^{1}\sin(k\pi t)\, \Big( \int_{0}^{1} { K}(t,s)d\omega_s
\Big) dt
 \\
  = - \int_{0}^{1}\, \Big( \int_{0}^{1} { K}(t,s) \sin(k\pi t) dt
\Big) d\omega_s  = - \frac{1}{k^2 \pi^2}  \int_{0}^{1} \sin(k\pi s)
d\omega_s =0.
\end{align}
However, the stochastic integral  $\frac{1}{k^2\pi^2}
 \int_{0}^{1}\sin (k\pi s )\,d\omega_s $
is a non-degenerate
  gaussian
random variable (with mean 0 and variance $ \frac{1}{k^4\pi^4}
\int_{0}^{1}\sin^2 (k\pi s )\,ds= \frac{1}{2 k^4\pi^4}$). It follows
 that the probability that (\ref{ortho}) 
  is verified vanishes.
  This   implies that
    (\ref{linear}) does not have a solution,
     for
$\sqrt\mu=k\pi$.

 \noindent Hence, we have proved

\begin{proposition}\label{fredstoc}
(i) If $\mu \neq n^2 \pi^2$,  the linear Dirichlet BVP associated to
(\ref{linear}) has a unique solution.
\par
\noindent (ii) If $\mu = m^2 \pi^2$ for some $m \ge 1$,  the linear
 Dirichlet BVP associated to (\ref{linear}) has no solution.
\end{proposition}

\begin{remark}  {\rm As in the deterministic case, the above result can be
also deduced from the explicit expression of the solution using
Fourier series. }
\end{remark}

 \begin{remark}  {\rm A standard argument shows that the above result
still holds in the general case
\begin{equation}\label{linear0}
f(t,X_t,X_t^\prime)=a X_t+b X_t^\prime ,\qquad a,b\in\rr,
\end{equation}
where the condition for the existence and uniqueness of the solution
of (\ref{eqnonres})  is now $a-b^2/4\ne k^2\pi^2$, with $k\in\zz$. }
\end{remark}

\subsection{Existence and uniqueness under Lipschitz-type
conditions}

 In this section we give some other existence and
  pathwise uniqueness results for our
 BVP,  taking into account  Proposition \ref{serve}
  and  using some tools of the theory of
 classical nonlinear ODEs. To this end, we
  will  consider the solution
 $Y$ (see \eqref{y1}).
  Let $\omega \in \Omega$ and define
 $\hat f: [0,1] \times \rr^2 \to \rr$ by
 $$ \hat f (t,x,y):=
f(t,x+Y_t(\omega),y+Y_t^\prime(\omega)).$$
 A straightforward computation leads to
\begin{lemma}\label{funge}
Let $\omega \in \Omega$ be fixed. A function
  $u \in C^1_0 $ is a solution
of
\begin{equation} \label{vecchia}
\left\{\begin{array}{l}
           u_t^\prime + \int_0^t f(s,u_s,u_s^\prime)  ds  = u_0^\prime  +
\omega_t\\
            u_0=0=u_1
            \end{array}
\right.
\end{equation}
if and only if  $z_t:= u_t-Y_t(\omega)$ belongs to $C^2([0,1])$ and
is a solution of
\begin{equation} \label{nuovaODE}
\left\{\begin{array}{l}
           z_t^{\prime\prime}+ \hat f (t,z_t,z_t^\prime)=0 \\
            z_0=0=z_1.
            \end{array}
\right.
\end{equation}
\end{lemma}

\noindent Note that, as a consequence of its definition, the function $\hat
f$ has the same regularity  of $f$ with respect to the second and
third arguments.

\noindent Lemma \ref{funge} allows to apply  the classical existence and
uniqueness results for boundary value problems by Bailey, Shampine
and Waltman \cite{BSW}. To do this, let $K,L$ be real numbers and define
\begin{equation} \label{DefAlfa}
\alpha(L,K)=\left\{\begin{array}{ll} \frac{2}{\sqrt{4K-L^2}}
\arccos \frac{L}{2\sqrt{K}} & \mbox{ if } 4K-L^2 >0
\\[6pt]
\frac{2}{\sqrt{L^2-4K}}\, \text{arccosh} \frac{L}{2\sqrt{K}} &
\mbox{ if } 4K-L^2 <0,L>0,K>0
\\[6pt]
\frac{2}{L} & \mbox{ if }
4K-L^2 =0,L>0
\\[6pt]
+\infty  & \mbox{ otherwise }
\end{array}
\right.
\end{equation}
\noindent and
\begin{equation} \label{DefBeta}
\beta(L,K)=\alpha(-L,K).
\end{equation}
The first result of \cite{BSW} that we use here is based on the
contraction mapping principle, and its proof consists in showing
the existence and uniqueness of a fixed point of an operator
defined through the Green's function for problem
$(\ref{nuovaODE})$ (analogue to the integral operator introduced in Section 2). However, more work is needed in order to
get an optimal result.

\begin{theorem} \label{PrimoLip} (\cite[Theorem 3.5]{BSW}).
Assume that there exist $K,L$ such that
\begin{equation}\label{IpoLipPrima}
|\hat f (t,x,y)-\hat f (t,\tilde x, \tilde y)| \leq K|x- \tilde
x|+L|y- \tilde y|,
\end{equation}
for all $t\in [0,1]$ and for all $x,\tilde x,y, \tilde y \in \rr$.
Assume also that $1<2\alpha(K,L)$.
 Then $(\ref{nuovaODE})$ has a unique solution.
\end{theorem}

\begin{remark}  {\rm The above result is optimal, in the sense
that neither
 existence nor uniqueness are guaranteed when
$1=2\alpha(K,L)$.

}
\end{remark}

 \noindent Recalling Proposition \ref{serve}
 and Lemma \ref{funge}, we obtain

\begin{corollary} \label{PrimoCor}
Assume that there exist $K,L$ such that
\begin{equation}\label{IpoPrimaCor}
|f (t,x,y)-f (t,\tilde x, \tilde y)| \leq K|x- \tilde x|+L|y-
\tilde y|,
\end{equation}
for all $t\in [0,1]$ and for all $x,\tilde x,y, \tilde y \in \rr$.
Assume also that $1<2\alpha(K,L)$. Then   \eqref{np} has a unique
solution. In particular, if
\begin{equation}\label{Inpartic}
|f (t,x,y)-f (t,\tilde x, \tilde y)| \leq L(|x- \tilde x|+|y-
\tilde y|),
\end{equation}
for all $t,x,\tilde x,y, \tilde y$ and $0<L<4$, then
  \eqref{np}  has a unique solution.
\end{corollary}

\begin{proof} It is sufficient to apply Theorem $\ref{PrimoLip}$,
Lemma $\ref{funge}$ and the definition of $\hat f$.  As for the
particular case
 when \eqref{Inpartic} holds, it is easy to check
 that if $0<L<4$ then we can get
\begin{equation}\label{alfa}
1< \frac{4}{\sqrt{4L-L^2}} \arccos
\frac{\sqrt{L}}{2}.
\end{equation}
From the definition of $\alpha$ it follows that the above inequality
is equivalent to $1<2\alpha(L,L)$ and thus Theorem  $\ref{PrimoLip}$
applies with $K=L$.
\end{proof}

 \noindent Corollary $\ref{PrimoCor}$ improves Proposition 1.4 in \cite{NP},
which shows existence and uniqueness under the assumption that
\begin{equation}\label{LipNP}
|f (t,x,y)-f (t,\tilde x, \tilde y)| \leq L(|x- \tilde x|+|y-
\tilde y|),
\end{equation}
for all $t\in [0,1]$ and for all $x,\tilde x,y, \tilde y \in \rr$,
and  $L<1/3$.

\noindent Corollary $\ref{PrimoCor}$ can be further improved by
means of a generalized Lipschitz condition. To this end, we recall

\begin{theorem} \label{SecondoLip} (\cite[Theorem 7.6]{BSW}).
Assume that $\hat f$ is locally Lipschitz and that there exist
$K,L_1,L_2$ such that
\begin{equation}\label{IpoLipSec}
\hat f (t,x,y)-\hat f (t,\tilde x, y) \leq K(x-\tilde x),
\end{equation}
for all $x \geq \tilde x, t\in [0,1], y\in \rr$,
\begin{equation}\label{IpoLipTerza}
L_1(y-\tilde y) \leq \hat f (t,x,y)-\hat f (t,x, \tilde y) \leq
L_2 (y-\tilde y),
\end{equation}
for all $y \geq \tilde y, t\in [0,1], x\in \rr$. Assume also that
$1<\alpha(L_2,K)+\beta(L_1,K)$. Then $(\ref{nuovaODE})$ has a
unique solution.
\end{theorem}

\noindent Arguing as above, we obtain

\begin{corollary} \label{SecondoCor}
Assume that $f$ is locally Lipschitz and that there exist
$K,L_1,L_2$ such that
\begin{equation}\label{IpoSecCor}
f (t,x,y)-f (t,\tilde x, y) \leq K(x-\tilde x),
\end{equation}
for all $x \geq \tilde x, t\in [0,1], y\in \rr$,
\begin{equation}\label{IpoTerzaCor}
L_1(y-\tilde y) \leq f (t,x,y)-f (t,x, \tilde y) \leq L_2
(y-\tilde y),
\end{equation}
for all $y \geq \tilde y, t\in [0,1], x\in \rr$. Assume also that
$1<\alpha(L_2,K)+\beta(L_1,K)$. Then $(\ref{vecchia})$ has a
unique solution.
\end{corollary}

 \noindent Corollary $\ref{SecondoCor}$ can be compared with
Proposition 1.3 in \cite{NP}, where it is assumed that $f=f(x,y)$ is
nonincreasing in each coordinate and that it has linear growth. More
precisely, the monotonicity condition in $x,y$ is contained in
$(\ref{IpoSecCor}),(\ref{IpoTerzaCor})$ when we take $K=0$ and
$L_2=0$, respectively. Moreover, it follows from the definitions
that $\beta(L_1,0)=+\infty$. Notice that no linear growth
restriction is required in Corollary $\ref{SecondoCor}$; the
assumptions are satisfied also (as remarked in \cite{BSW}) by a
nonlinearity of the form $f(t,x)=-e^x$.

\section {Uniqueness in law}

In this section we will  give sufficient conditions to have
 uniqueness in law for solutions to the BVP associated to
 equation \eqref{np}.
 These conditions are not covered  by the pathwise uniqueness results
 of previous sections.
  In this section (excluding    Remark \ref{exi})
 we will  always  assume that

\begin{hypothesis} \label{law}
The function $f : [0,1] \times \rr^2 $ is continuous and bounded and
has first and second spatial partial  derivatives $f_x$, $f_y$,
  $f_{xx}$, $f_{xy}$ and $f_{yy}$  which are continuous and bounded.
\end{hypothesis}

\subsection {$H-$differentiability }

 Let $H = L^2(0,1)$ and $H_0 $ be the subspace of $\Omega$
  introduced at the end of Section 1.
 Recall that a Hilbert-Schmidt operator $K : H \to H$  can be
represented
  by  a Kernel $K(t,s) \in L^2[(0,1)^2]$, i.e.,
  $
 K_t h = \int_0^1 K(t,s) h_s ds,$ $ t \in [0,1].
 $
 Identifying $H$ with $H_0$,
   a Hilbert-Schmidt operator $R: H_0 \to H_0$
 can be represented
  by  a Kernel $R(t,s) \in L^2[(0,1)^2]$  as follows:
  \begin{align} \label{f7}
R_t f =  \int_0^t dr\int_0^1 R(r,s) f_s' ds,\;\;\;
 f \in H_0, \;\; t \in [0,1].
 \end{align}
 {\it In the sequel we will identify Hilbert-Schmidt operators
  from $H_0 $ into $H_0$ with their  corresponding kernels in
   $ L^2[(0,1)^2]$; to stress this fact,  we will also  write
    $H_0 \otimes H_0 \simeq L^2[(0,1)^2]$.} The
     following definition is inspired from
\cite{S} (compare also with \cite[Chapter 4]{N},
  \cite[Section 3.3]{UZ1} and \cite[Definition B.6.2]{UZ1}).

\begin{definition} \label{Hdiff}
 Let $K$ be a real
 separable Hilbert space.
   A  measurable map ${\mathcal G} : \Omega \to K $   is said to be
   $H$-differentiable
  if the following conditions hold:

\vskip 1mm
 \noindent (1)  For any $\omega \in \Omega$, $\P$-a.s.,
 the mapping ${\mathcal G} (\omega + \cdot ) : H_0 \to
  K$, $h \mapsto {\mathcal G} (\omega + h )$,
    is Fr\'echet differentiable on $H_0$.

\noindent (2) For any $\omega \in \Omega$, $\P$-a.s.,
  the
 $H$-derivative
  $D_H {\mathcal G}(\omega)$, which is defined by
\begin{equation}\label{ramer4}
D_H {\mathcal G}(\omega)[h] = \lim_{r \to 0} \frac{ {\mathcal G}
(\omega +  r h) - {\mathcal G} (\omega)}{r}, \;\;\;\; h \in H_0,
\end{equation}  is a Hilbert-Schmidt operator from
 $H_0$ into $K$.

\noindent(3) the map
 $ \omega \mapsto \,     D_H {\mathcal G} (\omega ) $
 is measurable  from $\Omega$ into  $ H_0 \otimes K$.
\end{definition}


\begin{remark}
{\rm
 In condition (1)
 we are requiring that ${\mathcal G}$ is differentiable along the directions of
  $H_0$ (the
 Cameron-Martin space or the space of admissible shifts for $\P$,
  see \cite{UZ1}).
 The space $H_0$ is densely and continuously embedded in $\Omega$
 (the immersion $i: H_0 \to \Omega$ is even compact).
 The triple ($\Omega, H_0, \P$) is an important
 example of abstract Wiener
 space (see \cite[Section 4.1]{N}). The notion of
 $H$-differentiability can be more generally formulated in
  abstract Wiener spaces.
}
\end{remark}

\noindent In the special case when $K = H_0$ we obtain
 (see   \eqref{f7}
 and compare
with
  \cite[Theorem 2.1]{NP})

 \begin{definition} \label{Hdiff1} A
  measurable map ${\mathcal G} : \Omega \to H_0 $   is said to be
   $H$-differentiable
  if the following conditions hold:

\vskip 1mm
 \noindent (1)  For any $\omega \in \Omega$, $\P$-a.s.,
 the mapping ${\mathcal G} (\omega + \cdot ) : H_0 \to
  H_0$ is Fr\'echet differentiable on $H_0$.

 \noindent (2)  For any $\omega \in \Omega$, $\P$-a.s.,
  there exists the
$H$-derivative, i.e., a kernel $D_H {\mathcal G}(\omega) \in L^2
([0,1]^2)$, such that, for any   $\omega \in \Omega$,
 $\P$-a.s.,
\begin{equation}\label{ramer}
  \lim_{r \to 0} \frac{ {\mathcal G}
(\omega +  r h) - {\mathcal G} (\omega)}{r}=
  \int_0^{\, \cdot}  (D_H {\mathcal G}
(\omega)) [h'](s)ds, \;\;\; h \in H_0,
\end{equation}
 where
 $(D_H {\mathcal G} (\omega) ) [h'](t) =
 \int_0^1 D_H {\mathcal G} (\omega)(t,s)  h_s' ds$, $t \in [0,1].$

\noindent(3) the map
 $ \omega \mapsto \,     D_H {\mathcal G} (\omega ) $
 is measurable  from $\Omega$ into $ L^2 ([0,1]^2) $.
\end{definition}

\noindent The concept of $H$-differentiability goes back to Gross
 at the beginning of the 60s  and it is now
    well understood that it is strictly related to
   Malliavin Calculus (see also  Appendix A).
   The relation between the  $H$-differentiability and
  Malliavin derivative
  is completely  clarified in \cite{S}
   (see also \cite[Section 4.1.3]{N}).
  It turns out that
  $D_H {\mathcal G} $ is {\it the Malliavin derivative of
  ${\mathcal G}$}. More precisely,
  we have the following result as a special case of
  \cite[Theorem 3.1]{S}.

\begin{theorem} (Sugita \cite{S}) \label{sugita}
 Let $K$ be a real separable Hilbert space.
  Let us consider   a
    measurable map ${\mathcal G} : \Omega \to K $ which is
   $H$-differentiable and such that ${\mathcal G} \in L^2(\Omega; K)$
    and
$$
 D_H {\mathcal G} \in L^2(\Omega; H_0 \otimes K).
$$
 Then ${\mathcal G}$ belongs to $D^{1,2} (K)$
   (see Appendix A).
    Moreover,
      we have
    $D_M {\mathcal G}  = D_H {\mathcal G}$, $\P$-a.s..
\end{theorem}

\noindent Let us go back to the map $T$ given in \eqref{t1};
  $T : \Omega \to \Omega$, $T
= I + G$, where $G : \Omega \to H_0$,
\begin{equation} \label{ciao1}
 G_t(\omega) =  \int_0^t f (s, Y_s (\omega),
  Y_s'(\omega) ) ds, \;\;\; \omega \in
 \Omega,\;\; t \in [0,1].
 \end{equation}
We have the following lemma.

\begin{lemma} \label{fre} The following assertions hold:

\noindent (i) The mapping
 $T: \Omega \to \Omega$ is continuously
  Fr\'echet differentiable on $\Omega$,
  with Fr\'echet derivative $DT(\omega) : \Omega \to \Omega$,
\begin{align*}
 DT(\omega)[\theta] & =  \theta + \int_0^{\; \cdot}
  \Big ( f_x (s, Y_s (\omega), Y'_s(\omega)) \, Y_s (\theta)
  +  f_y (s, Y_s (\omega), Y'_s(\omega))Y_s' (\theta)  \Big)ds
\\ & = \theta + DG(\omega)[\theta],
   \;\; \omega, \, \theta \in \Omega.
\end{align*}
(ii)  The mapping
 $G: \Omega \to H_0$ is $H$-differentiable,
  with the following
 $H$-derivative $D_H G(\omega)$, for  any $\omega \in \Omega$,
 \begin{align*}
D_H G(\omega)[h](t) =  f_x (t, Y_t (\omega), Y_t'(\omega)) \, Y_t (
\tilde h)
 + f_y (t, Y_t (\omega), Y'_t(\omega)) \,
 Y_t' ( \tilde h)
\\
 =
 -   a_t (\omega) \int_0^1 K (t,s) h_s ds-    b_t (\omega) \int_0^1
\partial_t
 K(t,s)
h_s ds ,
   \;\; h \in H, \; t \in [0,1],
\end{align*}
where $a_t = a_t (\omega)= f_x (t,Y_t(\omega), Y_t'(\omega)) $
   and $b_t = b_t (\omega)= f_y (t,Y_t(\omega), Y_t'(\omega))$.
   Moreover, the following relation between
Fr\'echet and
 $H$-derivative holds:
 \begin{align} \label{cf}
D G (\omega)[h] (t) = \int_0^{t} D_H G(\omega)[h'](s) ds,
 \;\;\; h \in H_0,\;\; t \in [0,1],\;\; \omega \in \Omega.
 \end{align}
\end{lemma}
\begin{proof} (i) It is straightforward to check
 that $T$ is continuously
  Fr\'echet differentiable on $\Omega$.
  First one  verifies  its G\^ateaux-differentiability
   at a fixed $\omega$, finding
 the  G\^ateaux derivative $DT(\omega)$. The computations are easy,
  we only note the estimate
  $$
  \sup_{s, r \in [0,1]} |Y_s (\omega + r \theta)| \le \|
   \omega\|_{\infty} + \| \theta\|_{\infty}.
  $$
 Then
 one proves in a straightforward way that the mapping:
 $
 \omega \mapsto DT(\omega)
 $ from $\Omega$ into ${\cal L}(\Omega)$ (${\cal L}(\Omega)$
 denotes the Banach space of all linear and bounded operators from
  $\Omega$ into $\Omega$ endowed with the operator norm)
  is continuous
  and this gives the assertion.

\noindent  (ii)  First note that the operator
$$
 h \mapsto D_H G(\omega) [h]
= - a_t(\omega)\, \int_0^1 K (t,s) h_s ds -  b_t (\omega) \int_0^1
\partial_t
 K(t,s)
h_s ds ,
   \;\; h \in H,
$$ is  a Hilbert-Schmidt operator on $H$.
  To check the $H$-differentiability of $G$, it is enough
 to verify  that
(the limit is in $H$)
 \begin{equation}\label{ramer1}
 \lim_{r \to 0}  \frac{ G_t' \left(\omega + r \int_0^{\cdot} h_s ds \right) -
 G_t' (\omega) }{r} = f_x (t,Y_t (\omega),
 Y_t'(\omega)) \, Y_t ( \tilde h)
  + f_y (t,Y_t (\omega),
 Y_t'(\omega)) \, Y_t' ( \tilde h),
\end{equation}
 $h \in H$, where $\tilde h_t
 $ $= \int_0^{t} h_s ds$,
 and also that
 \begin{equation} \label{ci}
 h \mapsto \,     D_H G \left(\omega + \int_0^{\cdot} h_s ds \right)
 \;\; \text{ is continuous from } \; H  \; \text{into}\;
 L^2 ([0,1]^2),
\end{equation}
 for any  $\omega \in \Omega$.
  The proof of \eqref{ramer1} is straightforward
   (formula \eqref{ramer1} also appears  in
    \cite{NP}) and also the verification of
\eqref{ci}.

It remains to  show  the measurability property,
 i.e., that $\omega \mapsto D_H G (\omega)$ is measurable
  from $\Omega$ into $L^2([0,1]^2)$.
  We fix   an orthonormal
    basis $(e_i)$ in $H$ and consider
   the  orthonormal basis  $(e_i \otimes e_j)$  in
  $ L^2([0,1]^2) $; recall that
   $e_i \otimes e_j (t,s) = e_i(t)  e_j(s)$,
    $s,t \in [0,1]$ (cf. see \cite[Chapter VI]{RS}).
   To obtain the measurability property,
    it is enough to verify that, for any
    $i, j \ge 1$,  the  mapping:
     \begin{align} \label{measur}
 \omega \mapsto  \int_0^1 \int_0^1
  DF(\omega)(s,t) e_i(t)  e_j(s)
  dt ds
 \end{align}
 is measurable from $\Omega$ into $\rr$
  and  this follows  easily.
 The proof is complete.
\end{proof}

\begin{remark} We have, for any $\omega \in \Omega$,
\begin{align} \label{c4}
 \| D_H G(\omega) \|_{L^2([0,1]^2)} \le
 (\| f_x \|_0 + \| f_y \|_0)\, (
 \| K \|_{L^2([0,1]^2)} +
 \| \partial_t K \|_{L^2([0,1]^2)} ).
\end{align}
\end{remark}

\begin{lemma} \label{fre1}
 For any $\omega \in \Omega$,
  the  Fr\'echet
  derivative $DT(\omega) : \Omega \to \Omega $ is such that
 \begin{equation} \label{f}
 \begin{aligned}
D T (\omega) = I + D G(\omega) : \Omega \to \Omega
  \; \text {is an
isomorphism } \;
 \Leftrightarrow\;
\\  \text {the linearized equation}
\;\;
u_t'' + b_t u_t' +  a_t u_t  =0,\;\; u_0= u_1=0, \;\;\; \\
\text{with} \;\;  a_t = a_t (\omega)= f_x (t,Y_t(\omega),
Y_t'(\omega)),\; b_t = b_t (\omega)= f_y (t,Y_t(\omega),
Y_t'(\omega)),
\\ \text{has the unique zero solution.}
\end{aligned}
\end{equation}
\end{lemma}
\begin{proof} Since $DG (\omega)$ is a compact operator on
 $\Omega$, by the Fredholm alternative theorem it is enough
  to check that
 $ I + DG(\omega)
 $ is one to one. Fix $\omega$ and
  let $\theta \in \Omega$ be such that
 $$
 \theta_t +  \int_0^{t}
 \big( f_x ( s,Y_s(\omega),
Y_s'(\omega) ) \, Y_s (\theta) +
 f_y ( s,Y_s(\omega),
Y_s'(\omega) ) \, Y_s' (\theta) \big)  ds =0,\;\;\; t \in [0,1].
$$
It follows that  $\theta $ is differentiable and
$$
 \theta'_t + a_t (\omega)  \, Y_t (\theta)
  + b_t (\omega)  \, Y_t' (\theta)=0.
$$
Recalling that $\theta'_t = Y_t'' (\theta)$, we find that
 $Y_t (\theta) = u_t$ solves the boundary value  problem
 $ u_t'' +  a_t u_t +b_t u_t' =0,$  $u_0= u_1=0.$ Hence
  $Y (\theta) =0$ and so $\theta =0$.
\end{proof}

\subsection { An anticipative Girsanov  theorem
 involving a Carleman-Fredholm
 determinant
 }

 Here we present a non-adapted version of the Girsanov
  theorem proved
 recently in \cite[Theorem 3.3]{UZ}. This result will be used in the
 sequel to prove uniqueness in law for our boundary value problem
 \eqref{np}. Its formulation requires some concepts
  of Malliavin Calculus
  (see Appendix A). Recall that $H_0 \otimes H_0 \simeq L^2[(0,1)^2]$.

\begin{hypothesis} \label{uz} \

  (i) Let $F: \Omega \to H_0$ be a measurable mapping
 which belongs to $D^{2,2}(H_0)$.

(ii) If $\delta (F)$   denotes the  Skorohod
  integral of $F$ and $D_M F$ its Malliavin derivative,
   it holds
 \begin{align} \label{di}
 \exp \Big(  - \delta (F) +  \| D_M F \|_{L^2([0,1]^2)} \Big)
 \in L^{4}(\Omega).
  \end{align}
\end{hypothesis}

\noindent Let us comment the previous assumptions; (i) and (ii) are
immediately
 obtained from the
corresponding assumptions in
 \cite[Theorem 3.2]{UZ} with $r=2$ and $\gamma =3$.
 Consider  $\Lambda_F : \Omega \to \rr$,
 \begin{align} \label{lf}
 \Lambda_F(\omega) = \det{_{2}}(I + D_M F(\omega)) \,
  \exp \Big(  - \delta (F)(\omega) \, - \, \frac{1}{2}
  |F(\omega)|_{H_0}^2
  \Big).
\end{align}
As pointed out after \cite[Theorem 3.2]{UZ} (see also Appendix A.2
in \cite{UZ1}) under Hypothesis \ref{uz} we have $ \Lambda_F ,$ $
\Lambda_F (I + D_M F)^{-1} v  \in L^{4}(\Omega), $ for any $v \in
H_0$.

\begin{theorem}\label{ustu} (\"Ust\"unel-Zakai \cite{UZ})

(H1) Assume that $F$ satisfies Hypothesis \ref{uz} and consider the
 associated  measurable transformation ${\mathcal T}
  = {\mathcal T}_F : \Omega \to \Omega$,
 \begin{align} \label{t5}
 {\mathcal T} (\omega) = \omega + F(\omega),\;\;\; \omega \in \Omega.
 \end{align}

 (H2) Assume  that, for any $\omega \in \Omega$, $\P$-a.s.,
    $ [I + D_M F
\,(\omega)] : H_0 \to H_0 $
 is an isomorphism (here $I =I_{H_0}$).

(H3) Assume that there exists a measurable (left inverse)
 transformation
 ${\mathcal T}_l : \Omega \to \Omega $ such that
$$
{\mathcal T}_l ({\mathcal T} (\omega)) = \omega,\;\;\; \omega \in
\Omega, \; \P-a.s..
$$
 Then there exists a (Borel) probability measure $\Q$ on $\Omega$,
 which is  equivalent   to the Wiener measure
  $\P$, having density $\frac{d \Q}{d \P} = \Lambda_F$,
   and such that
\begin{align} \label{qq}
\Q({\mathcal T}^{-1}(A)) = \Q (\{ \omega \in \Omega\,:\, {\mathcal
T} (\omega) \in A \}) = \P (A), \;\;
 \text{for any Borel set} \;\;  A \subset \Omega.
\end{align}
\end{theorem}

\noindent Note that the assertion says that  the process $({\mathcal
T}_t(\omega))_{t \in [0,1]}$ is a Wiener process on
 $(\Omega, {\mathcal F}, \Q ).$
 The measure
 $\Q$ is called a  Girsanov measure in \cite{UZ}.

\begin{remark} \label{kusuoka} {\em It is useful to compare the previous
 theorem with another non-adapted
 extension of the Girsanov theorem known as the Ramer-Kusuoka
 theorem (see \cite{K}, \cite[Theorem 4.1.2]{N}
 and \cite[Section 3.5]{UZ1}).
  This result has been also applied in \cite{AN}, \cite{D}, \cite{DM} and \cite{NP}. Its
   formulation  requires  the following assumptions.
}

{\it
\smallskip (H1)
Assume that $F : \Omega \to H_0$ is $H$-differentiable
 and that the mapping:
 $
 h \mapsto \,     D_H F \left(\omega + h \right)
 $
 is continuous from $H_0$ into $H_0 \otimes H_0$,
  for any $\omega \in
\Omega$, $\P$-a.s..

(H2) Assume that  the
   measurable transformation ${\mathcal T}
   = I +F : \Omega \to \Omega$ (see \eqref{t5})
 is {\rm bijective.}

 (H3) Assume  that, for any $\omega \in \Omega$,
  $\P$-a.s.,  $ [I + D_H F
\,(\omega)] : H_0 \to H_0 $
 is an isomorphism.

If (H1)-(H3) hold,  then there exists a (Borel) probability measure
$\Q$ on $\Omega$,
 which is  equivalent   to
  $\P$, having density $\frac{d \Q}{d \P} = |\Lambda_F|$,
    such that \eqref{qq} holds.
}

\smallskip \noindent  {\em
 Note that Theorem \ref{ustu}
    does not require the invertibility of
${\mathcal T}$.
 On the other hand,   additional
 integrability assumptions on $F$ are imposed.
 There is also a difference in the expression of $\frac{d \Q}{d \P}$.
 Indeed  Theorem \ref{ustu} claims  that
  ${\det}_{2} (I + D_HF ) $
  is positive,
 $\P$-a.s., while in the Ramer-Kusuoka theorem, we have
 to consider $|{\det}_{2} (I + D_HF )|$. }
\end{remark}

\subsection {Some results on $H$-differentiability and
  Malliavin derivatives}

  Let $X = (X_t)$, $X :\Omega \to \Omega$ be a measurable
  transformation.
  We introduce an associated
 measurable mapping $S^X = S : \Omega \to \Omega$, as follows
 \begin{equation} \label{s}
 \begin{aligned}
& S_t(\omega) = \omega_t - \int_0^t f (s, X_s (\omega),
 X_s' (\omega)) ds =  [(I +
 F)(\omega)]_t, \;\; \text {where} \; F = F^X  :  \Omega \to H_0,
\\
& F_t (\omega) = - \int_0^t f  (s, X_s (\omega),
 X_s' (\omega)) ds,\;\; t \in [0,1].
\end{aligned}
 \end{equation}

 \begin{proposition} \label{rt}
  A measurable mapping $X : \Omega \to \Omega $
   is  a solution  if and only if there exists
    an admissible open set   $\Gamma \subset
 \Omega$, such
     that
 $$
 X_t(\omega) = Y_t (S(\omega)), \;\;\; \omega \in \Gamma,
  \;\; t \in [0,1].
 $$
\end{proposition}
\begin{proof} Recall that
 $
 Y_t (\omega) = - t \int_0^1 \omega_s ds +  \int_0^t \omega_s ds,
 $ so that
$$
Y_t (\omega) = \int_0^1 \partial_s \big (  t \wedge s \, - ts \big)
\, \omega_s ds.
$$
Let $X $ be a solution.  By Lemma \ref{equiv2} we have, for any
$\omega \in \Gamma$,
$$
 X_t (\omega) = \int_0^1 \partial_s \big (  t \wedge s \, - ts
 \big)  \, \big (\omega_s - \int_0^s f (r, X_r (\omega),
  X_r'(\omega)) dr \big)  \, ds
  =  Y_t (S(\omega)).
$$
The reverse implication follows similarly.
\end{proof}

\noindent Let us go back to the  continuous  map $T : \Omega \to \Omega$.
 Recall that pathwise uniqueness
   can be characterized by the fact that
  $T$ is {\it bijective}
  (see the precise statement in
   Proposition
  \ref{serve}).
   In this section
  we are mainly  interested in situations in
  which we do not know  if
  $T$ is
  bijective or not.

\noindent The following two results will be important. The first  one says that
 $T$ is always a measurable {\it left inverse of $S$} (compare with
   Theorem   \ref{ustu}).

 \begin{lemma} \label{1} Let $X$ be a solution to
   \eqref{integro}  and
 let $S$ be the associated measurable mapping (see \eqref{s}).
  We have on the admissible open set $\Gamma
   \subset \Omega$ (see \eqref{integro})
 \begin{equation}
 T \circ S = I;
  \end{equation}
in particular  $S$ is always {\it injective} on $\Gamma$ and $T$
{\it surjective} from $S(\Gamma)$ onto $\Gamma$.
 \end{lemma}
 \begin{proof}  We have, for any $\omega \in \Gamma$,
  using Proposition \ref{rt},
 \begin{eqnarray*}
 T_t (S(\omega)) =   S_t(\omega)  + \int_0^t f (s,Y_s (S (\omega)),
  Y_s' (S (\omega))) ds
 \\
= \omega_t - \int_0^t f (s,Y_s (S (\omega)),
  Y_s' (S (\omega))) ds + \int_0^t f (s, X_s (\omega), X_s'(\omega))
  ds =
\omega_t, \;\; t \in [0,1].
 \end{eqnarray*}
\end{proof}

\noindent We introduce now an assumption on solutions to  the
boundary value problem under consideration.
  Let $X$ be a solution to
\eqref{integro}.  We say
 that {\it $X$ satisfies the hypothesis (L)} if there exists an
  admissible Borel set $\Omega_0 \subset \Omega$ such that
\begin{align}
\label{ll}
\textbf {(L)} \begin{cases}
 \text{for any $\omega \in \Omega_0$,  the linearized BVP} \;
 u_t'' + b_t u_t' +  a_t  u_t =0,\;\; u_0= u_1=0,
 \\
\text{where} \;   a_t = a_t (\omega)= f_x (t,X_t(\omega),
 X_t'(\omega)) \; \text{ and} \;
b_t = b_t (\omega)= f_y (t,X_t(\omega), X_t'(\omega))
\\
\text{\it has only the zero solution.}
\end{cases}
\end{align}

\noindent If $T: \Omega \to \Omega$ is bijective (as it is always the case in \cite{NP})
 a
 condition which implies (L) is
\begin{align}
\label{ll1} \textbf {(LY)} \begin{cases}
 \text{for any $\omega \in \Omega$,
   the linearized BVP} \;
 u_t'' + b_t u_t' +  a_t  u_t =0,\;\; u_0= u_1=0,
 \\
\text{where} \;   a_t = a_t (\omega)= f_x (t,Y_t(\omega),
 Y_t'(\omega)) \; \text{ and} \;
b_t = b_t (\omega)= f_y (t,Y_t(\omega), Y_t'(\omega))
\\
\text{\it has only the zero solution.}
\end{cases}
\end{align}

\noindent Using Lemmas \ref{fre} and \ref{fre1} we can prove
 the following result (recall the admissible  open set
 $\Gamma \subset \Omega$
  given in \eqref{integro} and the fact that
  $T = I + G$ in (\ref{ciao1})).

\begin{theorem} \label{2} Assume Hypothesis \ref{law}.
  Let $X$ be a solution to \eqref{integro}
    which satisfies
  (L) and let $S = I + F $
   be the associated measurable mapping (see
 (\ref{s})).
\noindent  Then the  map $F$ is $H$-differentiable
 and we have, for any $\omega \in \Omega$, $\P$-a.s.,
 \begin{equation} \label{inv}
 [D_H F (\omega)] = [I + D_H G \,(S (\omega))]^{-1} - I
  = - D_H G (S(\omega))\, \big(I + D_H G (S(\omega)) \big)^{-1}.
\end{equation}
 Moreover, for any $\omega \in \Omega$, $\P$-a.s.
   (setting $I = I_{H_0}$),
$$
[I + D_H F \,(\omega)] : H_0 \to H_0 \;\;\; \text{is an
isomorphism.}
$$
 \end{theorem}
 \begin{proof}
 The proof is divided into some steps.

\noindent  {\it I Step. We show that
  there exists an admissible  open set $\Gamma_0
  \subset \Gamma $,  such that
  $S$  and $F$ are
  Fr\'echet differentiable  at any
  $\omega \in \Gamma_0$.
}

\smallskip \noindent According to formula \eqref{f}
  the Fr\'echet derivative  $D T (S (\omega))$ is an isomorphism
  from $\Omega $ into $\Omega$
  if and only if \eqref{ll} holds for $\omega$
   (recall that $X = Y \circ S$).
  Let  $\Omega_0 \subset  \Omega $
   be the admissible Borel set
   such that
   \eqref{ll} holds for any $\omega \in \Omega_0$. Define
    $\Omega' = \Omega_0 \cap \Gamma.$ Clearly $\P (\Omega') =1$
     and also $H_0 + \omega' \subset \Omega'$,
      for any $\omega \in \Omega'$, $\P$-a.s.. Thus $\Omega'$
       is an admissible Borel set in $\Omega$.

\noindent Fix $\omega \in \Omega'$.
  Since
  $D T (S (\omega))$ is an isomorphism,
  we can apply  the inverse function theorem and deduce
   that  $T $ is a local
  diffeomorphism  from  an open neighborhood $U_{S(\omega)}$
  of  $S(\omega)$ into an open neighborhood $V_{T(S(\omega))}=
   V_{\omega}$
  of  $T(S(\omega)) = \omega $. We may also assume that
   $V_{\omega} \subset \Gamma$, for any $\omega \in \Omega'$. Let us denote by $T^{-1}$  the local inverse function
     (we have $T^{-1} (V_{\omega}) = U_{S(\omega)}$).
  By Proposition
  \ref{rt}, we know that
$$
 \{ \theta \in \Gamma \, :\, S(\theta) \in T^{-1} (V_{\omega}) \} =
  V_{\omega}.
$$
It follows that $S$ is Fr\'echet differentiable in any
  $\omega' \in V_{\omega}$ and that
 $$
 D S(\omega' ) = (DT (S(\omega')))^{-1} = (I + DG (S(\omega')))^{-1}.
$$
 Introduce the open set
$$
 \Gamma_0 = \bigcup_{\omega \in \Omega'}
  \, V_{\omega} \subset \Gamma.
$$
Since $ \Omega' \subset \Gamma_0$, we have that $\P (\Gamma_0) =1$.
 In addition $H_0 + \omega \subset \Gamma_0$,
      for any $\omega \in \Gamma_0$, $\P$-a.s..

 \noindent The restriction of $S$ to $\Gamma_0$ is a
 Fr\'echet-differentiable
  function with values in $\Omega$. It follows that also
  $F$ is Fr\'echet differentiable at any
  $\omega \in \Gamma_0$ with Fr\'echet derivative
  \begin{align} \label{f4}
 DF(\omega) = (I + DG (S(\omega)))^{-1} - I.
 \end{align}
{\it II Step. We check  that, for any $\omega \in
  \Gamma_0 $,   $DF (\omega) [h] \in
 H_0$, if
 $h \in H_0$, and, moreover, for any $\omega \in
   \Gamma_0  $,
  $DF (\omega)\in H_0 \otimes H_0 $ (when considered as an
 operator from $H_0$ into $H_0$). We also show that,
    for any  $\omega \in \Gamma_0$, $\P$-a.s., the map:
  \begin{align} \label{lim}
  DF(\omega +  \cdot ) : H_0 \to H_0 \otimes H_0
  \;\; \text{is continuous} \end{align}
  and
  that $ DF(\cdot) $ is measurable from
 $\Gamma_0 $ into $H_0
  \otimes H_0$.
}

 \smallskip \noindent  Let us consider, for
  $\omega \in \Gamma_0$,   $k=
  (I + DG (S(\omega)))^{-1} [h]$. We have
   $k + DG (S(\omega))[k] =h$. It follows that $k \in H_0, $
    since $DG (S(\omega))[k] \in H_0$. By \eqref{cf}
     in Lemma \ref{fre}, we obtain that
  if $h \in H_0$, then
$$
(I + DG (S(\omega)))^{-1} [h] =
 (I + D_H G (S(\omega)))^{-1} [h].
$$
By using  the identity
$$
(I + D_H G (S(\omega)))^{-1} - I =
 - D_H G (S(\omega)) (I + D_H G (S(\omega)))^{-1}, \;\; \omega
  \in \Gamma_0,
$$
since $(I + D_H G (S(\omega)))^{-1}$ is a bounded operator and
 $D_H G (S(\omega))$ is Hilbert-Schmidt, we deduce that
$ (I + D_H G (S(\omega)))^{-1} - I  $
 is a Hilbert-Schmidt operator
on $H_0$ (see \eqref{hil}).
  We verify now the continuity property \eqref{lim},
i.e., that for any $\omega \in \Gamma_0$, $\P$-a.s.,
 for any $k \in H_0$,
$$
 \lim_{h \to k,\; h \in H_0} \!  D_H G (S(\omega + h ))
 (I + D_H G
 (S(\omega +h )))^{-1}
 \! \! =  D_H G (S(\omega + k)) (I + D_H G (S(\omega+k)))^{-1}
$$
(note that,
 for any
   $\omega \in \Gamma_0$, $\P$-a.s.,
    $DF(\omega + h)$ is well-defined  at any
 $h \in H_0$).
   This requires the
following considerations.

\vskip 1mm \noindent (a) The mapping: $D_H G : \Omega \to H_0
  \otimes H_0$ is continuous.  Indeed we know
 (see Lemma \ref{fre})
 $$
 D_H G (\omega) = - f_x (t,Y_t (\omega),
 Y_t'(\omega)) \, K(t,s)
  - f_y (t,Y_t (\omega),
 Y_t'(\omega)) \, \partial_t K(t,s)
$$
(identifying operators in $H_0 \otimes H_0$ with corresponding
kernels in $L^2([0,1]^2)$). Since $Y $ and $Y'$
 are continuous from $\Omega$ into $\Omega$ we get easily
  our assertion
  using  Hypothesis \ref{law}.

\vskip 1mm \noindent (b) Since   $S : \Gamma_0 \to \Omega$ is
   continuous  and $\Gamma_0$ is admissible,
    we get that,
    for any
   $\omega \in \Gamma_0$, $\P$-a.s., the map:
 $S(\omega \, + \, \cdot) : H_0 \to \Omega$ is continuous.
 Using also (a), we obtain that, for any
   $\omega \in \Gamma_0$, $\P$-a.s.,
  $(D_H G \circ S) (\omega + \cdot)  : H_0 \to H_0
  \otimes H_0$ is continuous.

 \vskip 1mm\noindent (c) To get the  assertion we  use \eqref{hil} and
 the following fact: for any
   $\omega \in \Gamma_0$, $\P$-a.s., we have
 $$
 \lim_{h \to k}  (I + D_H G
 (S(\omega+h)))^{-1}
  = (I + D_H G (S(\omega +k)))^{-1}
$$
(limit in ${\mathcal L}(H_0, H_0)$) for any $k \in H_0$. This
  holds since, for any
   $\omega \in \Gamma_0$, $\P$-a.s.,
    $(I+ D_H G (S(\omega +h )))$ is invertible
     for any $h \in H_0$,
     and,
   moreover, for any
   $\omega \in \Gamma_0$, $\P$-a.s.,  $\lim_{h \to k   }
  (I + D_H G
 (S(\omega+h)))
  $ $= (I+ D_H G (S(\omega +k)))$ in   ${\mathcal
  L}(H_0, H_0)$, for any $k \in H_0$.

\smallskip \noindent  To  check the measurability property,
  we can repeat the
  argument before formula \eqref{measur}.

\smallskip \noindent  {\it III Step.
    There exists   $c_0 >0$, depending on
 $\| f_x \|_0$ and $\| f_y\|_0$ such that, for any
$\omega \in \Gamma_0$,}
\begin{align} \label{c411}
|DS(\omega) h|_{H_0} =
  | (I + D_H G (S(\omega)))^{-1} h |_{H_0} \le
 c_0 |h|_{H_0},\;\;\; h \in H_0.
\end{align}
 This estimate
 follows from Corollary \ref{carle1} applied to
  $L = D_H G (S(\omega)). $

\smallskip \noindent {\it IV Step. We prove that $F$ is $H$-differentiable
 with $D_H F(\omega) = DF(\omega)$ (see \eqref{f7}),
  for any  $\omega \in \Gamma_0$.
}

\smallskip \noindent
  The assertion will be proved
  if we show that there exists,
  for any $\omega \in \Gamma_0$,  $R (\omega)
  \in H_0 \otimes H_0$,
  such that
 \begin{align} \label{gat1}
 \lim_{r \to 0} \frac{ {F} (\omega +  r h) - {F}
(\omega)}{r} = R(\omega) [h],\;\;\; h \in H_0
\end{align}
 (the limit is in $H_0$).
  Indeed, once this is checked we will get that
 $R(\omega) = DF(\omega)$ (because the topology
  of $H_0$ is stronger than the one in  $\Omega$).
 Moreover,
  we will obtain (since $\Gamma_0$
 is admissible) that, for any $\omega \in \Gamma_0$,
  $\P$-a.s., $F(\omega + \cdot ) : H_0 \to H_0$
   is G\^ateaux differentiable on $H_0$. Combining this fact with
   \eqref{lim}, we will deduce the required property
    (1) in  Definition \ref{Hdiff1}.

\noindent To  prove \eqref{gat1},
 we first show that,  for any $t \in [0,1]$,
 $\omega \in \Gamma_0$, and $h \in H_0$,
 \begin{align} \label{xt}
(i) \; \lim_{r \to 0} \frac{ {X_t} (\omega +  r h) - {X_t}
(\omega)}{r} = Y_t ( DS(\omega)[h] ),
\end{align}
$$
(ii) \; \; \lim_{r \to 0} \frac{ {X_t'} (\omega +  r h) - {X_t'}
(\omega)}{r} = Y_t' ( DS(\omega)[h] ).
$$
Let us only check (ii) (the proof of (i) is similar).
 Using  the fact that  $X = Y\circ S$ on $\Gamma_0$, we have
  (for $r$ small enough)
$$
 \frac{ {X_t'} (\omega +  r h) - {X_t'}
(\omega)}{r} = - \int_0^1 \Big ( \frac{ {S_s} (\omega +  r h) -
{S_s} (\omega)}{r} \Big) ds +
  \frac{ {S_t} (\omega +  r h) -
{S_t} (\omega)}{r}
$$
 and the assertion follows passing to the limit as $r \to 0$
  (using also \eqref{c411}).

\noindent Let us go back to \eqref{gat1}.
  Define, for $\omega \in \Gamma_0$, and $h \in H_0$,
$$
 R(\omega)[h](t) = \int_0^t
 \Big( a_s (\omega) Y_s ( DS(\omega)[h] )
 + b_s (\omega) Y_s' ( DS(\omega)[h] ) \Big) ds, \;\;
  t \in [0,1].
$$
We have
$$
\lim_{r \to 0}
\Big| \frac{ {F} (\omega +  r h) - {F} (\omega)}{r} -
R(\omega) [h] \Big|_{H_0}^2
$$
$$
= \lim_{r \to 0}
\int_0^1  \Big | - \frac{ f (s, X_s (\omega + rh),
X_s'(\omega + rh)) -  f (s, X_s (\omega), X_s'(\omega)) } {r} $$ $$
- a_s (\omega) Y_s ( DS(\omega)[h] )
 - b_s (\omega) Y_s' ( DS(\omega)[h] )\Big|^2 ds.
$$
Now an application of the dominated convergence theorem shows that
 the previous limit exists and is 0.
 The proof is complete.
\end{proof}

\noindent Next we provide   useful properties of the
Malliavin derivative of $F$, taking advantage of the techniques in
\cite{GGK} (see Appendix B).
 The first one is an $L^{\infty}$-estimate
  for $D_H F$ and will be important
  in Section 4.5.

\begin{proposition} \label{ne1}
 Under the assumptions of Theorem \ref{2},
 there exists $C>0$, depending on
 $\| f_x \|_0$ and $\| f_y\|_0$, such that,
  for any $\omega \in \Omega$, $\P$-a.s.
    (identifying $L^2([0,1]^2)$
 with $H_0 \otimes H_0$),
\begin{align} \label{c41}
 \| D_H F(\omega) \|_{L^2([0,1]^2)} \le
 C.
\end{align}
\end{proposition}
 \begin{proof}
  Using \eqref{hil},
   estimates \eqref{c4} and \eqref{c411} lead to the
     assertion.
 \end{proof}

 \noindent The following result  provides an ``explicit
 expression'' for
 the   Malliavin derivative  $D_H F$. The formula follows from
  \eqref{inv} and Theorem \ref{carle2}.

\begin{proposition} \label{esplic}
 Under the assumptions of Theorem \ref{2} (identifying
  $H_0 \otimes H_0$ with $L^2([0,1]^2)$), we have,
   for any $y \in L^2(0,1)$,   $\omega \in \Omega$,
    $\P$-a.s.,
$$
 D_H F(\omega)[y] = - \int_0^1 \gamma(t,s) y(s)ds,\;\;\; t
\in [0,1],
$$
with
$$
\gamma(t,s) = \begin{cases}
 ({\frac{1}{W}}) [a_t u_2(s)\psi(t)+b_t u_2(s) \psi'(t)],\;\;
  \;\; 0 \le s < t \le 1,  \\
   ({\frac{1}{W}}) (a_t u_2(t) + b_t u'_2(t)) \varphi(s),\;\;
  \;\; 0 \le t < s \le 1.
\end{cases}
$$
 Here $u_k, k =1,\, 2,$ denote the
solutions to $u''_k + b_t  u_k' + a_t u_k =0$
 (the coefficients $a_t$ and $b_t$ depend on $\omega$
  and are given in (\ref{ll}))
 with initial
conditions $u_1 (0) = u_2' (0) =1$, $u_1' (0) =  u_2 (0) =0$,
respectively. Moreover,
   $W=u_1u'_2-u_2 u'_1$,  $M=u_1(1)/u_2(1)$, and
 $$
 \varphi(s)=-u_2(s)M+u_1(s), \; \psi(t)=u_2(t)M-u_1(t),
 \;\; t \in
[0,1],\;\; s \in [0,1].
$$
\end{proposition}

\noindent  The next result is needed in Section 4.5.

\begin{proposition}
\label{dafare} Under the assumptions of Theorem \ref{2}, we have
that $F  \in D^{2,2}(H_0).$
\end{proposition}
\begin{proof} The proof is divided into some steps.

\smallskip \noindent  {\it I Step. We check that
 $G \in D^{2,2}(H_0)$.}

\noindent Since we already now that
  $G \in D^{1,2}(H_0)$, we only need to show   that
  $D_H G \in D^{1,2}(H_0 \otimes H_0). $ Applying Theorem \ref{sugita},
    it is enough to  prove that
   $D_H G : \Omega \to H_0 \otimes H_0$
    is $H$-differentiable and that
     $D_H(D_H G) \in L^{\infty} (\Omega, {\mathcal HS }
      (H_0, H_0 \otimes H_0)) $. We proceed similarly to the proof of Lemma \ref{fre}
 (with more involved
 computations).
  Recall that
  $H_0 \otimes H_0  \simeq L^2([0,1]^2)$. First we introduce
 a suitable operator $R(\omega) \in {\mathcal HS}(H_0, H_0 \otimes
  H_0)$, for any $\omega \in \Omega$. This operator
  can be  identified with
   an integral operator acting from $L^2(0,1)$ into
   $L^2([0,1]^2)$, i.e., with a
    kernel in $L^2([0,1]^3)$.
 For any
  $\omega \in
 \Omega$, we set
 $$ c_t = c_t(\omega)= f_{xx} (t,Y_t (\omega),
 Y_t'(\omega)), \;\;
 d_t= d_t(\omega)= f_{xy} (t,Y_t (\omega),
 Y_t'(\omega)),$$
  $$e_t=e_t(\omega)= f_{yy} (t,Y_t (\omega),
 Y_t'(\omega)).$$
Now $R(\omega) $ can be identified with the following
    kernel in $L^2([0,1]^3)$:
$$
 c_t K(t,s )K(t,r) +  d_t \, \partial_t K(t,s)\, K(t,r)
+  d_t  K(t,s) \, \partial_t K(t,r) + e_t \, \partial_t K(t,s)
 \, \partial_t K(t,r),
$$
  $t,s,r \in [0,1]$.  We have, for any
 $h \in H$, $\omega \in \Omega$,
 $$
 \lim_{r \to 0} \Big |
  \frac{ D_H G \left(\omega +
  r \int_0^{\cdot} h_s ds \right) -
 D_H G (\omega) }{r} - R(\omega)[h]
  \Big|_{L^2([0,1]^2)}=0,  \;\; h \in H.
$$
It is
     easy to check that
   $ h \mapsto \,  $ $   R \left(\omega +
    \int_0^{\cdot} h_s ds \right)$
  is continuous from  $ H $  into
  $L^2 ([0,1]^3)$,
 for any  $\omega \in \Omega$.
  In addition the mapping
    $\omega \to R (\omega)$
   is measurable from $\Omega$ into
 $L^2 ([0,1]^3)$ (this can be done using
 the argument before formula \eqref{measur}).
   This shows that
  $D_H G$
    is $H$-differentiable and moreover that
     $D_H^2 G (\omega) = R(\omega)$, $\omega \in \Omega$. Finally, it is easy to see that
 $D_H^2 G \in L^{\infty} (\Omega,
  L^2([0,1]^3)) $ (recall that $L^2([0,1]^3)$
       $ \simeq  {\mathcal HS }
      (H_0, H_0 \otimes H_0)).$

 \smallskip \noindent  {\it II Step. We prove that
  $D_H F$ is $H$-differentiable.    }

\noindent  In order to check
  condition (1) in Definition \ref{Hdiff}, we
  use the admissible open set $\Gamma_0 \subset \Omega$
    given in
   the proof of Theorem \ref{2}
    and
  prove that,
  for any
   $\omega \in \Gamma_0$, $\P$-a.s.,
   the mapping:
$$
 h \mapsto D_H F(\omega + h)
$$
from $H_0$ into $H_0 \otimes H_0$ is Fr\'echet differentiable
 on $H_0$.

    Let us consider a Borel set $\Omega'' \subset \Gamma_0$,
     with $\P (\Omega'')=1$ such that, for any
     $\omega \in \Omega''$, $\omega +H_0 \subset \Gamma_0$.
      Fix any
 $\omega \in \Omega''$.  We would like to differentiate
 in formula \eqref{inv}, i.e., to differentiate the mapping
 \begin{align} \label{g6}
 h \mapsto (I + D_H G (S(\omega +h)))^{-1} - I
 \end{align}
 from $H_0$ into $H_0 \otimes H_0$, applying the usual
 composition rules for   Fr\'echet derivatives.
  The only problem is that  the  mapping
 $h \mapsto S(\omega + h) = \omega + h + F(\omega + h)
  $ does not take values in $H_0$.
  This is the reason for which
   we will verify directly
 the Fr\'echet differentiability at a fixed $h_0 \in H_0$.
 By setting $(I + D_H G (S(\omega + h)))
= M (h)$, we have, for any $h \in H_0$,
$$
 {M^{-1}(h) - M^{-1}(h_0 )} =
  M^{-1}(h) \big( M(h_0 ) - M(h) \big)
 M^{-1}(h_0 )
$$
$$
= - M^{-1}(h) \big( D_H G ( [S(\omega + h)-
 S(\omega + h_0 )] + S(\omega+ h_0)) - D_H G (S(\omega+
 h_0))
  \big)
 M^{-1}(h_0 )
$$
$$
= - M^{-1}(h)  \Big( D_H^2 G ( S(\omega + h_0))
 \big[S(\omega + h)-
 S(\omega + h_0)  \big] \Big)
 M^{-1}(h_0 )$$
 $$ + \,  M^{-1}(h)  \, o ([S(\omega + h)-
 S(\omega + h_0)]) \, M^{-1}(h_0 )
$$ $$
= - M^{-1}(h)  \Big( D_H^2 G ( S(\omega + h_0))
 \big\{ (h-h_0) + D_H F(\omega +h_0)[h-h_0]  \big \} \Big)
 M^{-1}(h_0 )$$
 $$ - \,  M^{-1}(h)  \Big( D_H^2 G (S(\omega + h_0))
  \, [ o (h- h_0) ]  \Big)
 M^{-1}(h_0) $$ $$ + \, M^{-1}(h)  \, o ([S(\omega + h)-
 S(\omega + h_0)]) M^{-1}(h_0 ),
$$
 as $h \to h_0 $;  we have used  I Step together with
   the fact that $S(\omega + h)-
 S(\omega + h_0) = (h- h_0) + \big( F(\omega + h)- F(\omega+
  h_0) \big) \in H_0$  and
 $ S(\omega + h)-
 S(\omega + h_0)$ $=  (h- h_0) + D_H F(\omega
  + h_0)[h- h_0] + o(h-h_0)$ as $h \to h_0$.
   This shows the  Fr\'echet  differentiability of the mapping in
\eqref{g6} at
   $h_0$, with Fr\'echet derivative  along the direction
   $k \in H_0$ given by
 $$
 V(\omega)[k]= -  M^{-1}(h_0 ) \, \Big( (D^2_H G(S(\omega +
 h_0))
 \big[ k + D_H F(\omega + h_0)[k] \big] \Big)
    \, M^{-1}(h_0 ).
$$
Let $(e_j) $ be an orthonormal basis in $H_0$. Using
 \eqref{hil}, we find, for any $j \ge 1$,
 \begin{align} \label{che}
 \| V(\omega)[e_j]\|_{H_0 \otimes H_0} \le
 \| M^{-1}(h_0 )\|_{{\mathcal L}(H_0, H_0)}^2
  \big( \|  (D^2_H G(S(\omega + h_0))
 [e_j]\|_{H_0 \otimes H_0}
 \end{align} $$  + \, \|  D^2_H G(S(\omega + h_0))
\|_{{\mathcal L}(H_0, H_0 \otimes H_0)} \, | D_H F(\omega +
h_0)[e_j] |_{H_0}
  \big).
$$
 It follows  that, for any $\omega \in \Omega'',$
   $V(\omega)
\in
 {\mathcal HS}(H_0, H_0 \otimes H_0)$. Up to now we know
that condition (1) in Definition \ref{Hdiff} holds for $
 {\mathcal G}= D_H F$,
 with $D_H (D_H F)(\omega) = V(\omega)$,
   $\omega \in \Omega''$.
  It remains to check that $V(\cdot)$ is measurable from
  $\Omega''$ into ${\mathcal HS}(H_0, H_0 \otimes H_0)$. This holds
   if, for any $k \in H_0$,  the mapping:
$$
\omega \mapsto V(\omega)[k]
$$
is measurable from
  $\Omega''$ into ${\mathcal HS}(H_0, H_0)$ and  this is easy to
  check.
  The assertion is proved.

\smallskip \noindent  {\it III Step. We prove that
     $D_H(D_H F) \in L^{\infty} (\Omega, {\mathcal HS }
      (H_0, H_0 \otimes H_0)) $.}

\noindent By Theorem \ref{sugita}
       this will imply that $F \in D^{2,2}(H_0)$.
 Taking into account the bounds \eqref{c411} and
 \eqref{c41} and the fact that
  $D_H^2 G \in L^{\infty} (\Omega,
   {\mathcal HS }
      (H_0, H_0 \otimes H_0))) $,  we find
      (see  \eqref{che}), for any $\omega \in \Omega$,
       $\P$-a.s.,
$$
\| V(\omega)\|_{ {\mathcal HS}(H_0, H_0 \otimes H_0) }^2
 = \sum_{j \ge 1}\| V(\omega)[e_j]\|_{H_0 \otimes H_0}^2
 \,  \le \, C,
$$
 where $C>0$ depends on $\| f_x\|_0$,  $\| f_y\|_0$,
$\| f_{xx}\|_0$,  $\| f_{xy}\|_0$ and  $\| f_{yy}\|_0$. The proof
is complete.
\end{proof}

\subsection{ Exponential integrability of
 the Skorohod integral $ \delta (F)$}

We start with  a technical result from \cite[Section 3.1]{N} which
requires
 to introduce
  the space $L^{1,2}$ (see
   \cite[page 42]{N}).

\noindent A real stochastic process $u \in L^{2} ([0,1] \times \Omega)$
belongs to the class $L^{1,2}$ if, for almost all $t \in [0,1]$,
 $u_t \in D^{1,2}(\rr)$, and there exists a measurable version of
 the two-parameter process $D_M u_t $ which still belongs to
  $L^{2} ([0,1] \times \Omega)$. One can prove that $L^{1,2}
   \subset Dom(\delta)$. Moreover $L^{1,2}$ is a Hilbert space
    and has norm
$$
 \| u\|_{L^{1,2}}^2 = \| u\|^2_{L^2 ([0,1] \times \Omega) }
 + \| D_M u  \|^2_{L^2([0,1] \times \Omega)}.
$$
Let $u \in L^{1,2}$. Fix a  partition $\pi $ of $[0,1]$,
 $\pi = \{t_0 =0 < t_1 < \ldots< t_N =1 \}$. Let
  $|\pi| = \sup_{0 \le i \le N-1} |t_{i+1} - t_i|$ and define
the following random variable
$$
 \hat S^{\pi}(\omega) = \sum_{i=0}^{N-1} \frac{1}{t_{i+1} - t_i}
 \Big(\int_{t_i}^{t_{i+1}} \E \big[u_s /
 {\mathcal F}_{[t_i, t_{i+1}]^c} \big](\omega) \, ds \Big)
 \, (\omega(t_{i+1}) - \omega(t_i)),\;\; \omega \in \Omega,
$$
 $\P$-a.s.; here $\E \big[u_s /
 {\mathcal F}_{[t_i, t_{i+1}]^c} \big]$ denotes the conditional
 expectation of $u_s \in L^{2}(\Omega)$ with respect to
  the $\sigma$-algebra
 ${\mathcal F}_{[t_i, t_{i+1}]^c}$ (where $[t_i, t_{i+1}]^c
 = [0,1] \setminus [t_i, t_{i+1}])$. This is the $\sigma$-algebra
  (completed with respect to $\P$) generated by the random
  variables $\int_0^1 1_A(s) \, d\omega_s $,
   when $A$   varies over
   all Borel subsets of
  $[t_i, t_{i+1}]^c$ (see \cite[page 33]{N}).

\smallskip \noindent   According to \cite[page 173]{N}, when
   $u \in L^{1,2}$  there exists a sequence
  of partitions $(\pi^n)$ such that $\lim_{n \to \infty}
   |\pi^n|=0$ and
 \begin{align} \label{nua}
 \hat S^{\pi^n} \to \delta(u),\;\; \text{as} \; n \to \infty,
  \;\; \P-a.s.\;\;
  \text{and in } \;\; L^2(\Omega).
\end{align}
 We can now prove the following estimate.

\begin{proposition} \label{stima} Let $u \in L^{1,2} \cap
L^{\infty}([0,1] \times \Omega)  $. Then, for any
 $a >0$, we have
$$
 \E [\exp (  a \, |\delta (u)| \,)]
 \le 2 e^{\frac{a^2 \, \| u \|^2_{\infty}}{2}},
$$
 where $\| u \|_{\infty} = \| u \|_{ L^{\infty}([0,1] \times \Omega) }.$
\end{proposition}
\begin{proof}
 We will use   assertion \eqref{nua}, with the previous
 notation.
It is enough to prove the following bound, for any $n \ge 1$,
 \begin{align} \label{exp}
 \E [\exp (  a \, | \hat S^{\pi^n} | \,)]
 \le 2 e^{a^2 \, \frac{\| u \|^2_{\infty}}{2}}.
\end{align}
 Once \eqref{exp} is proved,
 an application of the
 Fatou lemma will allow us to get the assertion.

\smallskip \noindent By elementary
 properties of conditional expectation, we have, for almost
 all $s \in [0,1]$, $\omega $, $\P$-a.s.,
 $$
 |\E \big[u_s /
 {\mathcal F}_{[t_0, t_1]^c} \big] | \le \|u \|_{\infty},
 $$
 for any $0 \le t_0 < t_1 \le 1$. It follows that,
 for any $n \ge 1$, $\omega $, $\P$-a.s.,
$$
 \Big| \frac{1}{t_{i+1}^n - t_i^n}
 \int_{t_i^n}^{t_{i+1}^n} \E \big[u_s /
 {\mathcal F}_{[t_i^n, t_{i+1}^n]^c} \big](\omega) \, ds \Big|
  \le  \|u \|_{\infty}.
$$
 Setting $Z_{i ,n} = \frac{1}{t_{i+1}^n - t_i^n}
 \int_{t_i^n}^{t_{i+1}^n} \E \big[u_s /
 {\mathcal F}_{[t_i^n, t_{i+1}^n]^c} \big](\omega) \, ds$,
 we get
$$
\E [\exp (  a \, | \hat S^{\pi^n} | \,)] =
 \E \Big [ e^{a | \sum_{i=0}^{N_n} Z_{i ,n}
 (\omega(t_{i+1}^n) - \omega(t_i^n))  |  } \Big]
  \le \E \Big [ e^{a  \sum_{i=0}^{N_n} |Z_{i ,n}|
 |\omega(t_{i+1}^n) - \omega(t_i^n)  |  } \Big]
$$
$$
\le \E \Big [ e^{a  \| u \|_{\infty} \sum_{i=0}^{N_n}
 |\omega(t_{i+1}) - \omega(t_i)  |  } \Big]
 = \E \Big [ \prod_{i=0}^{N_n} e^{a  \| u \|_{\infty}
  |\omega(t_{i+1}^n) - \omega(t_i^n)  |  } \Big]
$$
$$
= \prod_{i=0}^{N_n} \E \Big [ e^{a  \| u \|_{\infty}
  |\omega(t_{i+1}^n) - \omega(t_i^n)  |  } \Big]
   = \prod_{i=0}^{N_n} \E \Big [ e^{a  \| u \|_{\infty}
  |\omega(t_{i+1}^n - t_i^n)  |  } \Big]
$$
(in the last step we have used the independence of increments
 and stationarity of the Wiener process). Now the bound \eqref{exp}
 follows easily, noting that
$$
 \E \big[ e^{c |\omega(t)|}  \big] \le 2e^{\frac{c^2 \, t }{2}},
  \;\;\; c>0,\;\; t \ge 0.
$$
Indeed, we have, for any $n \ge 1,$
$$
\E [\exp (  a \, | \hat S^{\pi^n} | \,)] \le
 \prod_{i=0}^{N_n} 2 \E \Big [ e^{\frac{a^2}{2}  \| u \|_{\infty}^2
  \, (t_{i+1}^n - t_i^n)    } \Big]
 = 2 e^{a^2 \, \frac{\| u \|^2_{\infty}}{2}}.
$$
\end{proof}
\noindent Identifying $F_t(\omega) = - \int_0^t f (s,X_s (\omega),
X_s'(\omega)) ds $,
 $t \in [0,1]$, with the associated  stochastic process
 $u \in L^{1,2}$
  $$
 u(t, \omega) = f (t, X_t (\omega), X_t'(\omega)), \;\; t \in [0,1], \;\;
  \omega \in \Omega
$$
(see also \cite[Section 4.1.4]{N}) and applying the previous result,
we obtain
\begin{corollary} \label{cia}
 Assume that $f : \rr \to \rr$ is a bounded
 function. Then,  for any $a > 0$, it holds:
\begin{align} \label{d67}
 \E [\exp (  a \, |\delta (F)| \,)]
 \le 2 e^{\frac{a^2}{2} \| f \|^2_0}.
 \end{align}
\end{corollary}

\subsection{The main results
}

 We state now our main result. This theorem implies
  as a corollary that
 uniqueness in law holds
 for our boundary value problem \eqref{np}
  in the class of solutions such that the corresponding
 linearized equations (see condition (L) in \eqref{ll})
  have  only the zero solution.
  {\it Hence uniqueness
  in law holds for \eqref{np} whenever all solutions
 $X$ to \eqref{np} satisfy (L).} For a concrete example, we refer to
 Section 4.6.

\noindent  We remark that
  a statement  similar to the  result below
   is given  in \cite[Theorem 2.3]{NP}
 {\it assuming in addition that there is
   pathwise-uniqueness}
   for the
 boundary value problem \eqref{np}.
 Indeed pathwise uniqueness and uniqueness for the linearized
 equation (see \eqref{ll1}) lead  by the Ramer-Kusuoka theorem
  (see Remark \ref{kusuoka}) to Theorem 2.3 in  \cite{NP}. More
  information on \cite[Theorem 2.3]{NP} are collected in Remark
 \ref{nual}.

\begin{theorem} \label{ab} Assume Hypothesis \ref{law}.
 Suppose that there exists  a  solution
 $X$ to \eqref{integro}
  such that (L) in \eqref{ll} holds.

\noindent Then there exists a probability measure
 $\tilde \Q$ on $(\Omega, {\mathcal F})$,
 which is equivalent to $\P$, having (positive $\P$-a.s.)
  density
\begin{align} \label{fr}
  \frac{ d \tilde \Q}{d \P} = \eta = \det{_{2}}(I + D_H G) \,
  \exp \Big(  - \delta (G) \, - \, \frac{1}{2} |G|_{H_0}^2
  \Big)
 \end{align}
 ($G$ is defined in (\ref{ciao1})),
  such that the law of $X$ under $\P$ is the same of $Y$ under
 $\tilde \Q$, i.e.,
  \begin{align} \label{00}
   \P (\omega \,:\, X^{}(\omega) \in A) =
 \P  (X \in A) = \tilde \Q  (Y \in A),\;\;\; A \in {\mathcal F}.
 \end{align}
\end{theorem}
\begin{proof}
   {\it Part I.}
 We verify applicability of Theorem  \ref{ustu} with
  $$
 {\mathcal T} := S = S^X
  $$
 ($S$ is defined in \eqref{s}  and $S = I + F$).
 First
 we have that hypothesis (H3) of Theorem \ref{ustu} holds
  with ${\mathcal T}_l = T $ by Lemma \ref{1}
   ($T$ is defined in \eqref{t1}).
Moreover, also (H2) holds by Theorem \ref{2}. It remains to check
 (H1), i.e., assumptions (i) and (ii) in Hypothesis \ref{uz}.
 Note that (i) holds by Corollary \ref{dafare}. The main point is to check (ii).  By \eqref{c41},
 we easily find that
$$
\exp \Big(  \| D_M F \|_{L^2}^2 \Big)
 \in L^{4}(\Omega).
$$
Thus to prove \eqref{di} it remains to check that $
 \exp (  - \delta (F) )
 \in L^{4}(\Omega)
$
and this follows from Corollary \ref{cia}.

\hh {\it Part II.}  We introduce  the measure  $\tilde \Q$
 and establish
  \eqref{fr} (without proving  the positivity of $\eta$).

 Recall
that Theorem \ref{ustu}
 says  that
\begin{align} \label{gir}
\P  ( A ) = \Q (S^{-1} (A)),  \;\;\;  A \in {\cal F},
\end{align}
where  $\Q$ is a probability measure  on
 $(\Omega, {\cal F})$, equivalent to $\P$, with
the following (positive $\P$-a.s.) density
 $$
 \frac{ d \Q}{d \P} = \Lambda_F = \det{_{2}}(I + D_M F) \,
  \exp \Big(  - \delta (F) \, - \, \frac{1}{2} |F|_{H_0}^2
  \Big);
 $$
 recall that $X = Y \circ S$, i.e.,
  $X_t (\omega)= Y_t
(S(\omega))$, $\omega \in \Omega$, $t \in [0,1]$,  and so
 (see \eqref{ciao1} and \eqref{s})
 $$
 F  = - G  \circ S.
$$

\hh  We denote by $\E ^{\P}$ and $\E^{\Q}$ the expectations with
respect to $\P$ and $\Q$.

\noindent Let $A \in {\mathcal F}.$ Introducing
 $\Lambda_F^{-1} : \Omega \to
 \rr_+$, where $\Lambda_F^{-1}(\omega) = \frac{1}{\Lambda_F^{}
 (\omega)}$ if $\Lambda_F(\omega) >0 $ and 0 otherwise (see
  \cite[Section 1.1]{UZ1}), we find
 \begin{eqnarray*}
 \P  (X \in A) = \P  (\omega\, :\, Y (S(\omega)) \in A  )
 = \P (\omega \, : \, S(\omega) \in Y^{-1} (A) )
\\
  = \E^{\P}  [ 1_{(S(\omega) \in Y^{-1} (A))}]
 = \E^{\P}  \Big[ 1_{(S(\omega) \in Y^{-1} (A))} \,
 \frac{d \Q }{d \P}
 \frac{d \P }{d \Q} \Big]
\\
= \E^{\Q} \Big [  1_{(S(\omega) \in Y^{-1} (A))} \, \frac{d \P }{d
\Q} \Big ] =
 \E^{\Q} \Big [
  1_{(S(\omega) \in Y^{-1} (A))} \, \Lambda_F^{-1} \Big ]
\\ =  \, \E^{\Q} \Big [  1_{(S(\omega) \in Y^{-1} (A))} \,
 (\det{_{2}}(I
+ D_H F ) )^{-1}   \,  \exp \Big (
  \delta (F) \, + \, \frac{1}{2} |F|_{H_0}^2
   \Big)    \Big].
\end{eqnarray*}
By the properties of the Carleman-Fredholm
 determinant (see \cite[Lemma A.2.2]{UZ1}), setting
  $R = D_H F(\omega)$,  $\omega \in \Omega$,
we know that
$$
(\det{_{2}}(I + R ) )^{-1}  = \det{_{2}}\big( (I + R )^{-1} \big)
\exp \big({\text{Trace}}(R^2  \, (I + R )^{-1}) \big),
$$
 where ${\text{Trace}}(R^2  \, (I + R )^{-1}) $ denotes the trace
  of the trace class (or nuclear) operator
   $R^2  \, (I + R )^{-1}$ (recall
   that the composition of two Hilbert-Schmidt operators is a trace
   class operator). Using \eqref{inv}, and the fact that
   ${\text{Trace}} (MN) =
    {\text{Trace}} (NM)$, for any Hilbert-Schmidt  operators
    $M$ and $N$, we get
\begin{eqnarray*}
\P  (X \in A)  = \E^{\Q} \Big [  1_{Y^{-1} (A)}(S(\cdot) ) \,
\det{_{2}}(I + D_H G (S (\cdot) )) \cdot
\\ \cdot\, \exp \Big(  {\text{Trace}}\big(
(D_H G)(S(\cdot))^2 \,
   (I + D_H G (S(\cdot)))^{-1}
  \big)
 \Big)  \,  \cdot \,  \exp
\Big (
  - \delta (G \circ S) \, + \, \frac{1}{2} |G \circ S|_{H_0}^2
   \Big)    \Big].
\end{eqnarray*}
 Now remark the law $\P_0$ of $S$ under $\P$, i.e.,
  $\P_0(A) = \P (S^{-1}(A))$, $A \in {\mathcal F}$, is equivalent
  to $\P$ by Theorem \ref{ustu} \cite[Lemma 2.1]{UZ}.
  Using this fact we can apply Theorem B.6.4 in \cite{UZ1} and obtain
  the following identity ($\P$-a.s. and so also $\Q$-a.s.)
 $$
 \delta (G \circ S) = (\delta (G)) \circ S - \langle
  G \circ S , F \rangle_{H_0} -
  {\text{Trace}}\big( (D_H G)(S(\cdot)) \, D_H F
  \big).
$$
$$
=(\delta (G)) \circ S - \langle
  G \circ S , F \rangle_{H_0} +
  {\text{Trace}}\big( (D_H G)(S(\cdot))^2  \,
    (I + D_H G (S(\cdot)))^{-1}
  \big).
$$
 We get, since $F = - G \circ S$,
\begin{align*}
\P  (X \in A) = \E^{\Q} \Big [  1_{Y^{-1} (A)}(S(\cdot) ) \,
\det{_{2}}(I + D_H G (S (\cdot) ))    \, \cdot
\\
 \cdot \,
\exp \Big ( - (\delta (G)) \circ S + \langle
  G \circ S , F \rangle_{H_0}
   \, + \, \frac{1}{2} |G \circ S|_{H_0}^2
   \Big)    \Big]
\\
= \E^{\Q} \Big [  1_{Y^{-1} (A)}(S(\cdot) ) \, \det{_{2}}(I + D_H G
(S (\cdot) ))    \, \cdot
\\
 \cdot \,
\exp \Big ( - (\delta (G)) \circ S - \langle
  G \circ S , G \circ S \rangle_{H_0} + \, \frac{1}{2}
   |G \circ S|_{H_0}^2  \Big )
\,
 \\  = \E^{\Q} \Big [  1_{Y^{-1} (A)}
 (S(\cdot) ) \, \det{_{2}}(I + D_H G (S
(\cdot) ))    \, \exp \Big ( - (\delta (G)) \circ S \, - \,
\frac{1}{2}
   |G \circ S|_{H_0}^2  \Big ).
\end{align*}
 The previous calculations show that
$$\det{_{2}}(I + D_H G (S (\cdot)
)) \, \exp \Big ( - (\delta (G)) \circ S \, - \, \frac{1}{2}
   |G \circ S|_{H_0}^2  \Big )  = \eta \circ S
     \in L^1(\Omega , \Q)
  $$ and that it is
    positive $\Q$-a.s. (or  $\P$-a.s.). Using that $\Q$ is
    a Girsanov measure (i.e., that the law of $S$ under $\Q$
     is  $\P$), it is is elementary to check that
 $\eta \in L^1 (\Omega, \P)$ and moreover
\begin{align} \label{f9}
\P  (X \in A)= \E^{\P} \Big [   1_{A}(Y)  \, \det{_{2}}(I + D_H G )
\,  \exp \Big ( - \delta (G)
   \, - \, \frac{1}{2} |G |_{H_0}^2
   \Big)    \Big].
\end{align}
Up to now we know that $\eta \in L^{1}(\Omega)$ and
 $\E^{\P}[\eta ]
=1$.

\hh {\it Part III.} It remains to show that $\eta >0$, $\P$-a.s.,
i.e., that
 $\gamma = \det{_{2}}(I + D_H G) >0$, $\P$-a.s.

\smallskip \noindent By Theorem \ref{ustu}, we know that
 $\det{_{2}}(I
+ D_H F ) >0$, $\P$-a.s. (or $\Q$-a.s.). This is equivalent to say
that  $\gamma \circ S >0$, $\P$-a.s.. Assume by contradiction
 that there exists $A \in {\mathcal F}$ with $\P (A) >0$ such that
 $\gamma(\omega) \le 0$, for any $\omega \in A$.
     We have
 $$
0 \ge  \E^{\P} [1_A \cdot \gamma ] = \E^{\Q} [1_A(S(\cdot) )
\gamma(S(\cdot)) ].
$$
 But $\E^{\Q} [1_A(S(\cdot) )
\gamma(S(\cdot)) ]$   is positive if $\Q (S^{-1}(A))>0$.
 This holds,
since
  $\E^{\Q} [1_{S^{-1}(A)} ]$ $
   = \E^{\Q} [1_{A}(S(\cdot)) ] = $ $\E^{\P}[1_A] >0 $. We have
    found a contradiction.
 The proof is complete.
\end{proof}

\noindent The assertion of the theorem implies that
  $ \det{_{2}}(I + D_H G)
 >0 $, $\P$-a.s..
 This means that under  the assumptions of Theorem
 \ref{ab} we have that  condition (LY) in  \eqref{ll1}
  holds  $\P$-a.s..

 \smallskip \noindent  Since $\eta$ in Theorem \ref{ab} does not depend on $X$, we get
  immediately

\begin{corollary} \label{dtt} Assume Hypothesis \ref{law}.
 Suppose that we have two solutions to \eqref{np},
 $X^1 $ and $X^2$, which both satisfy hypothesis (L)
  in \eqref{ll}.
  Then   $X^1$ and $X^2$  have the
 same
 law (i.e., for any Borel set $A \subset \Omega$, we have
  $
  \P (\omega \,:\, X^{1}(\omega) \in A)$ $
   =  \P (\omega \,:\, X^{2}(\omega) \in A)).$
\end{corollary}

\begin{remark} \label{nual}
{\em

 In \cite[Theorem 2.3]{NP} it is shown that
 the assertion of our Theorem \ref{ab} holds with
 $|\det{_{2}}(I + D_H G)|$ instead of
  $\det{_{2}}(I + D_H G)$
  if one assumes
that

 (i)  $f : \rr \times \rr \to \rr $ is of class $C^1$;

 (ii) $T$ is bijective;

 (iii) the condition of \cite[Proposition 2.2]{NP} holds
  (such  condition  guarantees the validity of
  \eqref{ll1}, for any $\omega \in \Omega$, and so it implies
   \eqref{ll},  for any $\omega \in \Omega$).

\noindent We point out that,
  in the notation of \cite{NP}, $\det{_{2}}(I + D_H G)$
 is written as  $\det{_{c}}(-D_H G)$).
   }
\end{remark}

\subsection {An application }

Here  we show
 an explicit stochastic boundary
value problem for which uniqueness in law holds,
 but it seems that no known method
 allows
   to prove pathwise uniqueness (see, among others, \cite{CH} and
  the seminal paper
  \cite{MW}). For such problem we can also establish existence of
   solutions (see also Remark \ref{exi} for a more general
    existence theorem). The result  looks similar to Theorem \ref{Nonres} (where we have
proved existence and pathwise uniqueness). However, note that here
{\it the non-resonance condition \eqref{ipononres}  can be violated
in a discrete set of points.}

 \begin{theorem}\label{Nonres1} Let us consider the boundary
 value problem \eqref{np}
 with $f(t,x,y) = f(x)$.
  Assume that $f \in C^2_b(\rr)$ and, moreover,
 that
\begin{align}\label{rese}
& (i) \; 0<  f^{\prime} (x) \le  \pi^2, \quad \text{for any}
 \; x\in
\rr;
\\ \nonumber
& (ii) \; A = \{ x \in \rr \, :\, f'(x) = \pi^2 \} \;\; \text{is
discrete}.
\end{align}
Then there exists a solution $X$.
 Moreover, uniqueness in law holds
for \eqref{np} (i.e., any  solution $Z$ of \eqref{np} has the same
law of $X$).
\end{theorem}
\begin{proof}   \underline{\it Uniqueness.}
    We will suitably  apply  Corollary \ref{dtt}. To this
purpose it is enough to show that any solution $X$ of \eqref{np}
verifies  condition (L), i.e.,
  there exists an
  admissible Borel set $\Omega_0 \subset \Omega$ such that
   $$
\label{123} (L) \begin{cases}
 \text{for any $\omega \in \Omega_0$,
  $\P$-a.s., the  BVP:} \;\;
 u_t'' +  f'(X_t(\omega))  u_t =0,\;\; u_0= u_1=0,
\\
\text{ has only the zero solution.}
\end{cases}
$$
Let us consider the following  set $\Omega_0  $:
$$
 \Omega_0 = \{ \omega \in \Omega \; :\;
  f'(X_t(\omega))   <  \pi^2, \;\;\;
     t \in
[0,1], \; a.e.\}.
$$
By looking at $\Omega \setminus \Omega_0$, it is not difficult
 to prove that $\Omega_0 $ is Borel.
  Note that $
 \Omega \setminus \Omega_0$ contains all
  $\omega \in \Omega $ such that
 there exists  an interval $I_{\omega} \subset [0,1]$
  on which $t \mapsto f'(X_t(\omega))   =  \pi^2$.
The proof is now divided into three steps.

\hh {\it I Step.} We show that, for any $\omega  \in \Omega_0$,
  (L) holds.

We will use
    the following well-known result (it is a straightforward
     consequence of \cite[Lemma 3.1, page 92]{CH}).
 Let $\rho_t$, $t \in [0,1]$,
  be a real and measurable
 function.  Assume that there exists $h >0$ such that
 $
  h <   \rho_t $ $ <   \pi^2,$ $   t \in
[0,1]$, {\em a.e..}   Then
 the linear boundary value problem
  $v_t'' +  \rho_t  v_t =0,\;\; v_0= v_1=0,$ has only the zero
  solution.

\noindent Let $\omega \in \Omega_0$. In order to apply the previous
result,
 we
 remark  that,
 \begin{align} \label{re}
 h_{\omega}  <  f'(X_t(\omega))   <  \pi^2, \;\;\;
     t \in
[0,1], \; a.e.,
\end{align}
 for some $h_{\omega} >0$. This follows,
 since $t \mapsto f'( X_t(\omega))$
  is continuous and positive on $[0,1]$.

\hh {\it II Step.} We show that $\P (\Omega_0)=1$.

\noindent
   Take any  $\omega \in
  \Omega \setminus \Omega_0$.
  There exists a time interval
   $I_{\omega} \subset [0,1]$ such that
$$
\pi^2   =  f'(X_t(\omega)), \;\;\;
     t \in I_{\omega}.
$$
By using the continuity of the mapping $t \mapsto
  f'(X_t(\omega))  $ and the fact that $A$ is
discrete, we infer that there  exists $x_{\omega}
 \in A$ such that $X_t(\omega) = x_{\omega}$, $t \in J_{\omega}$,
 for some time interval $J_{\omega}$ contained in $I_{\omega}.$
 This means that
$$
 \int_{0}^{1} K
(t,s) f(X_s(\omega))ds  + Y_t(\omega) = x_{\omega},
$$
for any $t \in J_{\omega}$ (see Lemma \ref{equiv2}).
 Differentiating
with respect to $t$, we get
$$
\int_{0}^{1} \frac{\partial K}{\partial t}(t,s) f(X_s(\omega)) ds
 = - Y_t'(\omega) =
      \int_0^1  \omega_s ds -    \omega _t, \;\;\;
       t \in J_{\omega}.
$$
 It is well-known that the map $\xi_t(\omega) =
 \int_{0}^{1} \frac{\partial K}
 {\partial t}(t,s) f(X_s(\omega)) ds $ belongs to $C^1([0,1])$.
 We have found
$$
 \omega _t = \int_0^1  \omega_s ds - \xi_t(\omega),\;\;\;
t \in J_{\omega}.
$$
On the right hand side, we have
 a function
  which is
 $C^1$ on $J_{\omega}$.
  This means that, for any $\omega \in \Omega \setminus \Omega_0$,
 there exists a time interval on which $\omega$
  is a $C^1$-function.
 Since the  Wiener process (see \eqref{wie}),
  $\P$-a.s., has trajectories which are never
 of bounded variation  in any time interval of $[0, 1]$,
   we have that
  $\P (\Omega \setminus \Omega_0)=0$.

\hh  {\it III Step. }
  We prove that, for any $\omega \in \Omega_0$,
 $\P$-a.s.,  we have $\omega + H_0 \subset \Omega_0$.

 Assume by contradiction that
  this is not true. This means that, there exists
  a Borel set  $\Omega' \subset \Omega_0$ with $\P (\Omega') >0$,
  such that, for any $\omega \in \Omega' $ there exists
   $h  \in H_0$ with $\omega + h \not \in \Omega_0$.
 Let us consider such $\omega$ and  $h$.

\noindent Arguing as before, we find that there exists  a time
 interval $J_{\omega +h} \subset [0,1]$  and some
  $x_{\omega +h }
 \in A$ such that $X_t(\omega +h) = x_{\omega+h}$,
 $t \in J_{\omega +h}$.
 This means that
$$
 \omega _t  + h_t = \int_0^1  \omega_s ds
 +  \int_0^1  h_s ds - \xi_t(\omega +h),\;\;\;
t \in J_{\omega +h}.
$$
 We have found that for each $\omega \in \Omega'$
 there exists a time interval on which $\omega$ is of bounded
  variation.
  This contradicts the fact that $ \P (\Omega') >0$ and
  finishes the proof of uniqueness.

\hh  \underline{\it Existence.} The proof is divided into three
steps.

\hh {\it I Step.} For any  $\omega \in \Omega$, consider the
sequence
  $(X^n(\omega)) $, with $X^1_t(\omega) = 0$, $t \in [0,1]$,
  and
   $$
 X^{n+1}_t(\omega) =
 \int_{0}^{1} K
(t,s) f(X_s^n(\omega))ds  + Y_t(\omega), \;\; n \ge 1,
 \;\; t \in [0,1].
 $$
Using the boundedness of $f$,
 an application of the Ascoli-Arzel\`a theorem shows that,
  for any $\omega \in \Omega$, there exists a subsequence
 $(X^k(\omega))$ (possibly depending on $\omega$) which converges in
 $C([0,1])$ to a continuous function
  $X(\omega)$. It is then clear that, for any
 $\omega \in \Omega$, we have
  \begin{equation} \label{cr}
 X_t(\omega) =
 \int_{0}^{1} K
(t,s) f(X_s(\omega))ds  + Y_t(\omega), \;
 \;\; t \in [0,1].
\end{equation}
 The main difficulty is that the previous construction
  does not clarify  the measurable dependence
 of $X$ on $\omega$. To this purpose we will suitably
 modify $X$ in order to obtain
  the required measurability property.

\hh {\it II Step.} We investigate when
 condition  (LY) in \eqref{ll1} holds, i.e.,
  for which $\omega
 \in \Omega $
 \begin{align}
\label{ll12} \begin{cases}
 \text{
   the linearized BVP:} \;
 u_t'' +   f' (Y_t(\omega))  u_t =0,\;\; u_0= u_1=0,
\\
\text{ has only the zero solution.}
\end{cases}
\end{align}
  Arguing as in the proof of uniqueness,
    condition \eqref{ll12} holds in particular if  $\omega$ satisfies
 \begin{align} \label{re2}
 h_{\omega}  <  f'(Y_t(\omega))   <  \pi^2, \;\;\;
     t \in
[0,1], \; a.e.,
\end{align}
 for some $h_{\omega} >0$. On the other hand, if \eqref{re2}
   does not hold for $\omega^0 \in \Omega$, then
  there exists   $x_{\omega^0}
  \in A$ such that $Y_t(\omega^0) = x_{\omega^0}$,
   $t \in J_{\omega^0}$,
 for some time interval $J_{\omega^0} \subset [0,1]$.
  It follows that
 $
  Y_t(\omega^0) = x_{\omega^0},
 $ for any $t \in J_{\omega^0}$.
 Differentiating
with respect to $t$, we get
$$
  0 =
      -\int_0^1  \omega_s^0 ds +    \omega _t^0, \;\;\;
       t \in J_{\omega^0}.
$$
 This implies that $\omega _t^0 = \int_0^1  \omega_s^0 ds$,
  $t \in J_{\omega^0}$. Let us introduce the set
   $\Lambda \subset \Omega$ of all $\omega$ such that
   there exists a time interval $I_{\omega} \subset [0,1]$
    on which $\omega$ is a function of bounded variation.
     It is not difficult to prove that
    $\Lambda $ is a Borel subset of $\Omega$. Moreover,
$
 \P (\Lambda ) =0.
$

  We have just verified that
  \eqref{ll12} holds for any $\omega \in
   \Omega \setminus \Lambda $.

\hh {\it III Step.} Let us consider the mapping  $X(\omega)$
 of Step I and   introduce $S : \Omega \to \Omega$,
$$
 S_t(\omega) = \omega_t - \int_0^t f (X_s (\omega)) ds.
$$
We have $X(\omega) = Y(S(\omega))$ and $T(S(\omega)) = \omega$, for
any $\omega \in \Omega$ as in Section 4.3.
 Although $S$ {\it is not
necessarily measurable,}
   one can easily
   check that
 $$
 S^{-1} (\Lambda) =   \Lambda.
$$
This implies that  $S (\Omega \setminus \Lambda) = \Omega \setminus
 \Lambda $ (clearly $\P (\Omega \setminus \Lambda)=1$).
  Now we argue as in the proof of  Theorem \ref{2} with its notations.
  Since we know that
  \eqref{ll12} is  verified when
  $\omega = S(\theta)$, for some
 $\theta \in \Omega \setminus \Lambda$, we deduce that
  the Fr\'echet derivative  $D T (S (\omega))$ is an isomorphism
 from $\Omega $ into $\Omega$, for any
 $\omega \in \Omega \setminus \Lambda$.

\noindent
   By the inverse function theorem, $T $ is a local
  diffeomorphism  from  an open neighborhood $U_{S(\omega)}$
  of  $S(\omega)$ to an open neighborhood $V_{T(S(\omega))}=
   V_{\omega}$
  of  $T(S(\omega)) = \omega $, for any $\omega
   \in \Omega \setminus \Lambda$.
    Let us denote by $T^{-1}$  the local inverse function.
    We  deduce that, for any $\omega
   \in \Omega \setminus \Lambda$,
 $ S(\theta) = T^{-1}(\theta),\;\;\; \theta \in V_{\omega}.
 $

 \noindent Introduce the open set $$
 \Phi = \bigcup_{\omega \in \Omega \setminus \Lambda}
  \, V_{\omega}.
$$
 Since $ \Omega \setminus \Lambda \subset \Phi $, we have that $\P
(\Phi) =1$. In addition $\Phi$
  is an admissible open set in $\Omega$, since, for
 any  $ \omega \in \Omega \setminus \Lambda$, we have that
  $\omega  + H_0 \subset \Omega \setminus \Lambda
   \subset \Phi.$

 The restriction of $S$ to $\Phi$ is a
 $C^{1}$-function with values in $\Omega$. We define
  the measurable mapping
$$
 \hat S : \Omega \to \Omega, \;\;\; \hat S (\omega)
  = \begin{cases}S(\omega),\;\;\; \omega \in \Phi \\
    0,\;\;\; \omega \in \Omega \setminus \Phi
   \end{cases}
$$
 and  introduce
$ \hat X : \Omega \to \Omega $, $
 \hat X_t (\omega) = Y_t(\hat S(\omega)),\;\;\; \omega \in \Omega,
  \;\; t \in [0,1].
$

\noindent It is clear that $\hat X$ is measurable. Moreover, since
  $\hat X (\omega) = X(\omega)$, when $\omega \in \Phi$, we have
  that $\hat X$  verifies \eqref{cr} for any $\omega
   \in \Phi$. This shows that $\hat X$ is a solution to
   \eqref{np} and finishes the proof.
\end{proof}

\noindent An example of
 $f$ which is covered by the previous result  is
 $$f(x) =
   {\pi^2 } \int_0^x e^{-t^2}
 dt, \;\; x \in \rr.
 $$

\begin{remark} \label{exi} {\em The  previous
  proof  shows that an existence result for \eqref{np}
 holds, more generally, if the following three conditions hold:

\hh (i) \; $f(t,x,y) = f(x)$ with $f \in C_b(\rr) \cap
 C^1(\rr)$;

\hh (ii) \; there exists a Borel set $\Lambda \subset \Omega$
 such that $\Omega \setminus \Lambda $ is admissible and, moreover,
 $ S (\Omega \setminus \Lambda ) \subset \Omega \setminus
 \Lambda  $, where $S : \Omega \to \Omega$ is  defined
 by
$$
 S_t(\omega) = \omega_t - \int_0^t f (Z_s (\omega)) ds,
  \;\;\; \omega \in \Omega, \;\; t \in [0,1],
$$
 where $Z : \Omega \to \Omega$ is  any mapping (non necessarily
  measurable);

\hh (iii) condition \eqref{ll12} holds, for any $\omega \in
 \Omega \setminus \Lambda.$

\noindent Under (i)-(iii), the existence of solution
 can be proved by
 adapting the proof of Theorem \ref{Nonres1}.
}
\end{remark}

\section {Remarks on   computation of the Carleman-Fredholm
 determinant $\det_2(I + D_H G)$   }

\medskip

When dealing with non-adapted versions of the Girsanov theorem (see \cite{K}, \cite{R}\cite{UZ}) one delicate problem
 is to find some explicit expression for the Carleman-Fredholm
 determinant appearing also in  \eqref{lf} of
 Section 4.2. This problem has been also considered  in
 \cite{D}, \cite{DM}, \cite{NP}
 and \cite{UZ1} for different measurable transformations
  $\mathcal T$.
    In particular the Radon-Nykodim  derivative appearing in
      Theorem \ref{ab} and
       \cite[Theorem 2.3]{NP}
  (see also  Remark
  \ref{nual})
  contains
 the explicit term
$$
 \det{_{2}}(I + D_H G(\omega))
$$
 (in the notation of \cite{NP}, $\det{_{2}}(I + D_H G(\omega))$
  becomes $\det{_{c}}(-D_H G(\omega))$).

 \noindent The assertion in our next result is a reformulation of
\cite[Lemma 2.4]{NP}. It provides  an explicit formula for
 $\det{_{2}}(I + D_H G(\omega))$. It is important to point out
that our computation of the Carleman-Fredholm determinant $\det_2(I+
D_H G(\omega))$ has been developed with techniques which are
completely different from those (based on Malliavin calculus) used
for the proof of Lemma 2.4 in \cite{NP}.

\noindent Our approach comes from \cite{GGK} and it uses functional analysis
and the theory of linear ordinary differential equations. For the
reader's convenience, we have collected in Appendix B some of the
ideas (taken from \cite{GGK}) which have enabled us to perform our
computation of the Carleman-Fredholm
 determinant and some important consequences of
this approach.

\noindent We believe that this method could be useful in other situations
 (cf. \cite{AN}, \cite{D}, \cite{DM}, \cite{UZ1}).

\begin{lemma} \label{new} Assume that $f \in C^1$ and that
  the linearized BVP
 $$
u_t''  + b_t(\omega)u_t' +   a_t (\omega) u_t =0,\;\; u_0= u_1=0,
 $$
where $a_t = f_x (t, Y_t(\omega), Y_t'(\omega) )$,
 $b_t = f_y (t, Y_t(\omega), Y_t'(\omega) )$,
has the only zero solution,
 for any
 $\omega \in \Omega$.

\noindent Then the  following relation holds
\begin{equation*}
{\det}_{2} (I + D_H G(\omega) )  = Z_1(\omega) \,  \exp \Big
(\int_0^1 (t a_t + (1-t)b_t)
 dt \Big),
\end{equation*}
where  $Z_t$ solves the Cauchy problem
 $$
 u_t''  + b_t u'_t + a_t u_t =0,\;\; u_0=0,\;\; u'_0=1.
 $$
\end{lemma}

\begin{proof}
 The proof is based on some ideas which are developed in Appendix B.
More precisely, observe that, by \eqref{BElle}, the assumption in
Lemma \ref{new} guarantees that we can apply  Theorem \ref{carle0}
with
 $L = D_H G(\omega)$.
\end{proof}

\section*{Appendix A: Some definitions  of Malliavin Calculus}

 Here we summarize some notions of Malliavin
Calculus (see \cite{N}, \cite [Chapter V]{IW} and \cite[Appendix B]{UZ1}.

\noindent If $K$ and $M$ are real separable Hilbert spaces,
  we
consider the
 tensor product of $K$ and $M$, i.e. $K \otimes M$
  (this is the real Hilbert space
  formed by all Hilbert-Schmidt operators from $K$ into
  $M$;  see \cite[Chapter VI and (48) in page 220]{RS}).
   We also use the notation ${\mathcal HS}(K, M)$ for
$K \otimes M$.

\noindent Moreover, for $k \in K$ and
  $h \in M$, we consider  the linear operator
   $k \otimes h$ from $K$ into $M$:
  $$
 (k \otimes h) (u) := \langle k, u\rangle_K h,\;\;\; u \in K.
$$
Let $H_0$ be the Hilbert space introduced at the end of Section 1.

The smooth $K$-valued functionals on $(\Omega, H_0, \P)$ are
functionals $a: \Omega \to K$ of the form
$$
a(\omega) = \sum_{i=1}^N f_i (\langle h_1, \omega {\rangle}, \ldots,
{\langle}h_m, \omega {\rangle} ) k_i,
$$
where $f_i \in C^{\infty}_b (\rr^m)$, $h_1, \ldots, h_m \in H_0$,
 $k_1, \ldots, k_N \in K$, and we set
 $$
 {\langle}h, \omega {\rangle} = \int_0^1 h_s' d\omega_s \;\; \text{
 (It\^o integral)},
  \;\;\; h \in H_0, \;\; h'= \frac{dh}{ds}.
$$
For smooth $K$-valued functionals we define the Malliavin derivative
$$
 D_M a(\omega) = \sum_{i=1}^N
 \sum_{j=1}^m
  \partial_{x_j} f_i ({\langle}h_1, \omega {\rangle}, \ldots, {\langle}h_m, \omega
{\rangle} )  h_j \otimes  k_i.
$$
The Sobolev  space $D^{1,2}(K) \subset L^2(\Omega, K)$ is now the
completion of smooth $K$-valued functionals with respect to the norm
$$
 \| a\|_{1,2} = \| a\|_{L^2(\Omega, K)} + \| D_M a \|_{L^2(\Omega ;
  H_0 \otimes
  K)}.
$$
 The Malliavin derivative  on $D^{1,2}(K)$  (still denoted by
   $D_M$) is  the closure of
 $D_M $ as defined on
 smooth $K$-valued
 functionals.

\noindent When $K = \rr$, the adjoint of $D_M$ is denoted by $\delta
$ and is
 called the Skorohod integral. Hence if $\xi \in  L^2 (\Omega , H_0)$,
  we say that $\xi \in \text{dom}(\delta)$ if
 we have
 $$
 \E [ {\langle}D_M \phi, \xi{\rangle}_{H_0} ]
 \le \| \phi\|_{L^2} \, c(\xi),\;\;
$$
for any $\phi \in D^{1,2}(\rr)$. If $\xi \in \text{dom}(\delta)$, we
have
 $\delta \xi \in L^2(\Omega, \rr)$ and
$$
\E [ {\langle}D_M \phi, \xi{\rangle}_{H_0} ] = \E [\phi \, \delta
(\xi) ].
$$
 We also need to introduce the second Malliavin derivative.
 Let $F: \Omega \to H_0$ be a measurable mapping
 which belongs to $D^{1,2}(H_0)$. If
  $D_M F \in D^{1,2}(H_0 \otimes H_0) $
 then we say that $F \in D^{2,2}(H_0)$ and set $D_M^2 F =
  D_M (D_M F)$ .
 Note that, for any $\omega$, $\P$-a.s.,
$$
D_M^2 F(\omega) \in  \big( H_0 \otimes (H_0 \otimes H_0) \big)
 = {\mathcal HS}(H_0, H_0 \otimes H_0).
$$

\section*{Appendix B:
 An input-output representation for linear boundary
 value problems  }

\noindent In this section, we briefly sketch the framework of
\cite[Chapter XIII]{GGK} in which our computation of the
Carleman-Fredholm determinant (Lemma \ref{new}) is developed.
 Throughout this section, since only deterministic functions
 are involved,    we go back to the notation
  $\alpha(t)=\alpha_t$, for any real function
 $\alpha$.
  We are
concerned with the Hilbert-Schmidt integral operator defined as
follows:
 \begin{align}\label{elle}
  (L h)(t) =  - a(t)\,  \int_0^1 K (t,s)  h(s) ds -
 b(t) \int_0^1 \frac{\partial
K}{\partial t}  (t,s) h(s) ds, \;\; h \in H,\; t \in [0,1].
 \end{align}
 Here $a$ and $ b$ are given real continuous functions on $[0,1]$ and
  $$
  K(t,s)=t\wedge s-ts
  $$  is the Green's function of $-d^2/dt^2$
(with Dirichlet boundary condition).

\noindent From the definition of the integral operator $L$ in
\eqref{elle}, it is easy to check
that for any  $y \in H$ a function $\xi$ solves
\begin{align} \label{4}
\xi''(t) + b(t) \xi'(t) + a(t)\xi(t)  = y(t), \;\; \xi(0)=\xi(1)=0
\end{align}
if and only if, setting $u:=\xi''$, it is $(I+L)u=y$.
In other
words,
\begin{equation}\label{BElle}
{\rm Problem \,  (\ref{4}) \, \,  is \,  solvable} \;
 \Longleftrightarrow \; {\rm
   the \, operator} \,\, (I +L) : H \to H \,
{\rm \, is \, invertible.}
\end{equation}
Note that the equation in \eqref{4} can be rewritten as
\begin{equation} \label{IO}
  \left\{ \begin{array}{l}
 u=\xi^{\prime\prime} \
\\
 y = a(t)\xi + b(t)\xi^\prime+u,
 \\
\xi(0)=\xi(1) = 0.
\end{array} \right.
\end{equation}

\noindent In \cite[Section XIII]{GGK}, \eqref{IO} is called an
 input-output representation of $(I+L)$, where $u$ is the input and $y$ is the output. More precisely, setting $\xi=x^1$, $\xi'=x^2$, \eqref{IO} is of the form
\begin{equation}\label{rr1}
  \left\{ \begin{array}{l}
 x' = Ax + Bu
\\
 y = C(t) x + u
 \\
 N_1 x(0) + N_2 x(1) = 0,
\end{array} \right.
\end{equation}
with
 $$
  N_1=\left(
\begin{array}{cc}
1 &   0 \\
0 &  0 \\
\end{array}
\right), \; N_2=\left(
\begin{array}{cc}
0 &   0 \\
1 &  0 \\
\end{array}
\right), A=\left(
\begin{array}{cc}
0 &   1 \\
0 &  0 \\
\end{array}
\right), B=\left(
\begin{array}{c}
 0 \\
  1 \\
\end{array}
\right), C(t) = (a(t), b(t)). $$ It can be easily verified that the
inverse $u = (I +  L )^{-1} y$ admits the following representation
\begin{equation}\label{rr2}
  \left\{ \begin{array}{l}
 x' = (A - BC(t) ) x + By
\\
 u = - C(t) x + y
 \\
 N_1 x(0) + N_2 x(1) = 0,
\end{array} \right.
\end{equation}

\noindent i.e.

 \begin{equation}
  \left\{ \begin{array}{l}
 \xi^{\prime\prime} = -a(t)\xi-b(t)\xi^\prime+y
\\
 u = -a(t) \xi -b(t) \xi^\prime+y,
 \\
\xi(0)=\xi(1) = 0.
\end{array} \right.
\end{equation}
We now introduce the fundamental matrices $U^\times $,
\begin{align} \label{f5}
\frac{dU^{\times}}{dt}(t) = (A -BC(t) ) {U^{\times}}(t),\;\;\;
{U^{\times}}(0)=I, \;\;
 i.e.,\; \; {U^{\times}}(t) = \left(
\begin{array}{cc}
u_1 (t) &   u_2 (t) \\
u_1'(t) &  u_2'(t) \\
\end{array}
\right),
\end{align}
where $u''_k + b(t)  u_k' + a(t) u_k =0$,
$k =1,\, 2$, $u_1 (0) = u_2'
(0) =1$, $u_1' (0) =  u_2 (0) =0$, and
$$ U(t) = \left(
\begin{array}{cc}
1 &   t \\
0 &  1 \\
\end{array}
\right).
$$

\noindent With the previous notation, one can prove

\begin{proposition} \label{carle1}
Assume that $(I + L)$ is invertible. Then there exists a positive
constant $C$ (depending only on the coefficients $a$ and
 $b$ through their supremum norms
 $\| a\|_0$ and   $\| b\|_0)$
  such that
$$
 | (I + L)^{-1} y |_H \le C | y|_H,\;\;\; y \in H.
$$
\end{proposition}
 \begin{proof} We make straightforward estimates on the
 control problem
 \eqref{rr2} based on the Gronwall lemma.
\end{proof}


One can also deduce from \cite[Theorem XIII.5.1]{GGK} the next
result, which leads to Proposition \ref{esplic} in Section 4.

\begin{theorem} \label{carle2}
With the previous notation, assume that $(I + L)$ is invertible.
Then
\begin{equation}\label{uy}
u(t) = (I+L)^{-1}y(t) = y(t) - \int_0^1 \gamma(t,s) y(s)ds,\;\; t
\in [0,1],
\end{equation}
where with $P^{\times} = (N_1 + N_2 U^{\times}(1))^{-1} N_2
U^{\times}(1) $,
$$
\gamma(t,s) = \begin{cases}
 C_t U^{\times}(t) (I - P^{\times}) U^{\times}(s)^{-1} B,\;\;
  \;\; 0 \le s < t \le 1,  \\
   -C_t U^{\times}(t)  P^{\times} U^{\times}(s)^{-1} B,\;\;
  \;\; 0 \le t < s \le 1,
\end{cases}
$$
 or, more explicitly,
$$
\gamma(t,s) = \begin{cases}
 ({\frac{1}{W}}) [a(t) u_2(s)\psi(t)+b(t) u_2(s) \psi'(t)],\;\;
  \;\; 0 \le s < t \le 1,  \\
   ({\frac{1}{W}}) [a(t) u_2(t) + b(t) u'_2(t)] \varphi(s),\;\;
  \;\; 0 \le t < s \le 1,
\end{cases}
$$
 where $u_1$ and $u_2$ are introduced in \eqref{f5},
   $W=u_1u'_2-u_2 u'_1$, $M=u_1(1)/u_2(1)$ and
 $$
 \varphi(s)=-u_2(s)M+u_1(s), \; \psi(t)=u_2(t)M-u_1(t),
 \;\; t \in
[0,1],\;\; s \in [0,1].
$$
\end{theorem}

\noindent Finally by \cite[Theorem XIII.7.1]{GGK}, we obtain
 \begin{theorem} \label{carle0}
Assume that $(I+L)$ is invertible. Setting
$$
P = (N_1 + N_2 U^{}(1))^{-1} N_2 U^{}(1),
$$
we have
\begin{equation*}
\begin{split}
{\det}_2 (I + L) =& \det ( I - P + P U(1)^{-1}\, U^{\times}(1) )\,
e^{\int_0^1 \text{\rm tr} (C_t U(s) P U^{-1}(s) B)ds }
\\
=& u_2 (1) \, \exp \Big (\int_0^1  (t a(t) + b(t))
 (1-t)dt \Big),
 \end{split}
\end{equation*}
where $u''_2 + b_t  u_2' + a_t u_2 =0$,  $ u_2' (0) =1$, $ u_2 (0)
=0$.
\end{theorem}

\end{document}